\documentclass[11pt]{article}
\usepackage[utf8]{inputenc}
\usepackage{multirow}
\usepackage{tikz}
\usetikzlibrary{positioning}
\usepackage[a4paper, margin=2.3cm]{geometry}
\usepackage{amsmath,amssymb,amsthm,mathtools,mathrsfs}
\usepackage{thmtools}
\usepackage{thm-restate}

\usepackage{xcolor}
\usepackage{parskip}
\usepackage{float}
\usepackage[hidelinks,hyperfootnotes=false]{hyperref}
\usepackage{tikz-cd}
\usetikzlibrary{matrix,arrows.meta}

\usepackage{setspace}
\usepackage{makecell}

\usepackage{graphicx}
\usepackage{enumitem}
\usepackage{comment}

\newcommand{\F}{\mathbb{F}}
\newcommand{\C}{\mathbb{C}}
\newcommand{\1}{\mathbf{1}}

\newcommand{\Q}{\mathcal{Q}}

\usepackage{booktabs,tabularx,needspace}

\newtheorem{theorem}{Theorem}[section]
\newtheorem{remark}[theorem]{Remark}
\newtheorem{corollary}[theorem]{Corollary}

\newtheorem{lemma}[theorem]{Lemma}

\newtheorem{proposition}[theorem]{Proposition}

\newtheorem{conjecture}[theorem]{Conjecture}

\newtheorem*{definition*}{Definition}
\newtheorem*{remark*}{Remark}

\makeatletter
\def\theHtheorem{\thesection.\arabic{theorem}}
\@for\ARevEnv:=remark,corollary,problem,lemma,question,proposition,definition,conjecture,example,claim,assumption\do{
  \expandafter\def\csname theH\ARevEnv\endcsname{\theHtheorem}}
\makeatother

\newcommand{\Pin}{\Delta_{\mathrm{pin}}}

\usepackage[T1]{fontenc}
\usepackage{lmodern}
\allowdisplaybreaks[2]
\usepackage{etoolbox}
\BeforeBeginEnvironment{proof}{\Needspace{5\baselineskip}}
\BeforeBeginEnvironment{conjecture}{\Needspace{10\baselineskip}}
\hypersetup{pdftitle={On the finite field spherical restriction conjecture in four dimensions: the sharp endpoint and applications},pdfauthor={Thang Pham and Boqing Xue}}
\title{On the finite field spherical restriction conjecture in four dimensions: the sharp endpoint and applications}
\author{Thang Pham \thanks{Institute of Mathematics and Interdisciplinary Sciences, Xidian University. \newline
\hspace*{0.45cm} Email: {\tt thangphammath@xidian.edu.cn}} \and Boqing Xue \thanks{Institute of Mathematical Sciences, ShanghaiTech University. ~Email: {\tt xuebq@shanghaitech.edu.cn}}}
\date{}

\begin{document}

\maketitle

\begin{abstract}
Let $p$ be an odd prime. We prove the sharp extension estimate $R_{S_j}^*(2\to r)\lesssim_r 1$ for every sphere of nonzero radius $S_j\subseteq\mathbb{F}_p^4$ and every $r\geq3$, uniformly in $p$ and $j$. The proof combines arithmetic Hecke operator bounds with a refined fourth moment and an orthogonal decomposition. We also formulate a localized spherical restriction/extension conjecture that predicts the sharp dependence on the size of the physical support.
This conjecture remains open and would yield distance estimates for almost every pin at the conjectured Erd\H{o}s--Falconer exponent in four dimensions, with an arbitrarily small power loss in the size hypothesis. Unconditionally, our localized estimates imply that, for every $\varepsilon>0$, sets of size at least $p^{7/3+\varepsilon}$ determine
$(1+o_\varepsilon(1))p$ distances from almost every pin, together with asymmetric two-set versions.
\end{abstract}

\tableofcontents

\section{Introduction}\label{sec:main}

\subsection{Background and main results}

Let $q$ be an odd prime power and $\F_q$ be the finite field with $q$ elements. Let $\chi$ be the canonical nontrivial additive character of $\F_q$. We identify the dual of $\F_q^n$ with $\F_q^n$. All ambient norms on $\F_q^n$ use counting measure. Thus,
\[
 \|g\|_{L^r(\F_q^n)}:= \Bigg(\sum_{y\in \F_q^n}|g(y)|^r\Bigg)^{1/r},
 \qquad 1\le r<\infty,
\]
with the usual modification when $r=\infty$.

Let $V\subseteq \F_q^n$ be a nonempty algebraic variety. We write $d\sigma$ for the normalized surface measure on $V$, so that
\[
 \int_V f\,d\sigma
 = \frac{1}{|V|}\sum_{x\in V}f(x).
\]

For a function $f$ on $V$, its extension operator is defined by
\[
 (f\,d\sigma)^\vee(x)
 :=\frac{1}{|V|}\sum_{y\in V} f(y)\chi(y\cdot x),
 \qquad x\in \F_q^n.
\]
For a family of varieties $V$, the finite field Fourier extension problem asks for which exponents $1\le u,r\le \infty$ one has the estimate
\begin{equation}\label{extension}
 \|(f\,d\sigma)^\vee\|_{L^r(\F_q^n)}
 \le C \|f\|_{L^u(V,d\sigma)}
\end{equation}
for all functions $f$ on $V$, with a constant $C$ independent of $q$ and $f$. We denote this estimate by $R_V^*(u\to r)\lesssim_{u,r} 1$.

By duality, \eqref{extension} is equivalent to the restriction estimate
\begin{equation}\label{restriction}
 \|\widehat{g}\|_{L^{u'}(V,d\sigma)}
 \le C\|g\|_{L^{r'}(\F_q^n)},
\end{equation}
where $u'$ and $r'$ are the corresponding conjugate exponents and
\[
\widehat{g}(\xi):=\sum_{x\in\F_q^n}g(x)\chi(-\xi\cdot x),\qquad \xi\in\F_q^n.
\]

Mockenhaupt and Tao \cite{MT04} introduced the finite field Fourier restriction/extension problem. They studied several model varieties over $\F_q^n$, including paraboloids, cones and flat disks, and established connections between restriction estimates and Kakeya-type incidence geometry. Their work led to many further results. See, for instance, \cite{silva, chen, JF, IK08, AK, Kohkang, KK2025, KKY, kohmz, KohPhamVinh2021, Le13, european, Mle, Le15, RS}.

The sphere is a natural and difficult case of this problem. For $j\in \mathbb{F}_q^\times$, let $S_j^{(n)}$ be the sphere defined by $Q_n(x)=j$, where $Q_n(x):=x_1^2+\ldots+x_n^2$ for $x=(x_1,\ldots,x_n)\in \F_q^n$. By $f\lesssim g$ or $f=O(g)$, we mean that $|f|\leq C g$ for some absolute constant $C>0$. If both $f\lesssim g$ and $g\lesssim f$ hold, we write $f\approx g$.

In a 2008 paper, Iosevich and Koh \cite{IK08} proved that
\[R_{S_j^{(n)}}^*(2\to r)\lesssim 1,\]
for all $r\ge \frac{2n+2}{n-1}$. This result is sharp in odd dimensions when the estimate is required to hold uniformly over all nonzero radii \cite[Lemma~1.13]{KohPhamVinh2021}, and the exponent $\frac{2n+2}{n-1}$ is known as the $L^2$ Stein--Tomas exponent. In even dimensions, the conjecture asserts that
\begin{equation}\label{ConjEven}
R_{S_j^{(n)}}^*(2\to r)\lesssim 1 \iff r\ge \frac{2n+4}{n}
\quad\mbox{for even dimensions } n\ge 2.
\end{equation}

In the same paper, Iosevich and Koh established the conjecture in two dimensions by a direct geometric argument. Their approach does not extend to higher even dimensions, where the conjecture has remained open. Theorem~\ref{thm:main-endpoint-three} below settles the four-dimensional case over prime fields.

From now on, write $\F:=\F_p$ for an odd prime $p$ and $S_j:=S_j^{(4)}$. We denote the normalized surface measure on $S_j$ by $d\sigma_j$.

Our main result gives the sharp $L^2\to L^r$ extension range for
spheres of nonzero radius in $\F^4$.

\begin{theorem} \label{thm:main-endpoint-three}
We have $R_{S_j}^*(2\to r)\lesssim_r1$ for every $r\ge3$.
\end{theorem}

The exponent $3$ is sharp, which follows from an affine isotropic line $\ell\subseteq S_j$.
Such a line exists for every nonzero $j$. See the construction in
Proposition~\ref{prop:lower-bounds-Lambda}. Since $|S_j|=p^3-p$,
the extension of $\1_\ell$ has modulus $p/|S_j|$ on $p^3$ physical
points and vanishes elsewhere. Thus, the ratio of its $L^r(\F^4)$ norm
to $\|\1_\ell\|_{L^2(S_j,d\sigma_j)}$ is comparable to $p^{3/r-1}$,
which is uniformly bounded only if $r\ge3$.

We compare this result with two known estimates. Over the reals,
the $L^2\to L^r$ extension estimate for the sphere in $\mathbb{R}^4$ holds
precisely in the Stein--Tomas range $r\ge10/3$, and this range is sharp
by the Knapp example (see \cite{Tomas75} and \cite[Chapter~IX]{Stein93}).
Thus, the improvement in Theorem~\ref{thm:main-endpoint-three} is
specific to the finite field setting. Rudnev and Shkredov~\cite[Theorem~1]{RS}
proved the sharp $L^2\to L^3$ extension estimate for the paraboloid in
$\F_p^4$. Our result establishes the same endpoint for nonzero spheres,
uniformly over all odd primes and all nonzero radii. For the
four-dimensional paraboloid over arbitrary finite fields of odd
characteristic, Iosevich, Koh and Lewko~\cite[Theorem~1.1]{IKL20}
proved the range $r\ge28/9$, with no restriction on the square class
of $-1$.

The proof uses the operator structure of the Kloosterman-type kernel
of the sphere. We express the relevant quadratic-root matrices as
weighted sums of finite Weil operators, then identify these sums with
compressions of automorphic Hecke operators. This identification
preserves the coefficient spaces and their Hilbert-space norms.
Drinfeld's theorem gives the operator bound needed for a refined
fourth-moment estimate. An orthogonal decomposition then separates
the large Fourier coefficients and proves the strong endpoint without
a logarithmic loss.

The proof also gives a localized estimate that records the
restriction norm in terms of the size of the physical support. We
formulate the expected sharp bound in Conjecture~\ref{conj_main}.
This localized conjecture remains open. It would imply the
Erd\H{o}s--Falconer distance conjecture, with a conclusion for almost
every pin and an arbitrarily small power loss in the size hypothesis.

To state this conjecture precisely, we introduce the following notation.

For a non-empty algebraic variety $V\subseteq \F^4$ and $0\leq \vartheta\leq 4$, define $\Lambda_V(\vartheta)$ to be the infimum of $\lambda\geq 0$ such that, for any \(E\subseteq\F^4\) with \(1\leq |E|\leq p^\vartheta\) and any function \(f:\, V\to\mathbb{C}\),
\begin{equation}
\label{eq:localized-horizontal-extension}
\|(f\,d\sigma)^\vee\|_{L^2(E)} \leq p^\lambda \|f\|_{L^2(V,d\sigma)}.
\end{equation}
By duality, the quantity $\Lambda_V(\vartheta)$ is also the infimum of $\lambda\geq 0$ such that, for any \(E\subseteq\F^4\) with \(1\leq |E|\leq p^\vartheta\) and any function \(g:\F^4\to\mathbb{C}\) supported on
\(E\),
\begin{equation}
\label{eq:localized-horizontal-restriction}
\|\widehat{g}\|_{L^2(V,d\sigma)} \leq p^\lambda \|g\|_{L^2(\F^4)}.
\end{equation}
Both infima are attained. When $V=\F^4$, Plancherel's theorem gives $\Lambda_{\F^4}(\vartheta)=0$ for all $0\leq \vartheta\leq 4$.

\begin{conjecture}[Localized Spherical Restriction/Extension Conjecture] \label{conj_main}
As $p\rightarrow \infty$ through the odd primes,
\[
\sup_{j\in \F^\times}\Lambda_{S_j}(\vartheta)\rightarrow \Lambda_\flat(\vartheta)
\]
uniformly for $0\leq \vartheta\leq 4$, where
\[
\Lambda_\flat(\vartheta):=
\begin{cases}
0,&0\le\vartheta<2,\\[1mm]
\dfrac{\vartheta-2}{2},&2\leq \vartheta\leq 3,\\[2mm]
\dfrac{1}{2},&3<\vartheta\le4.
\end{cases}
\]
\end{conjecture}

Proposition~\ref{prop:lower-bounds-Lambda} gives $\sup_{j\in \F_p^\times}\Lambda_{S_j}(\vartheta)\geq \Lambda_\flat(\vartheta)$ for all $0\leq \vartheta \leq 4$. Thus, the conjectured bound is best possible in every range. Proposition~\ref{prop:conjecture_equiavlence} shows that Conjecture~\ref{conj_main} holds if and only if
\[
\sup_{j\in \F^\times}\Lambda_{S_j}(2) \rightarrow 0,\qquad p\rightarrow \infty.
\]

Therefore, it is enough to prove the conjecture at $\vartheta=2$.

We prove the following unconditional localized estimate.

\begin{theorem}\label{thm_main_Lambda}
For any $\varepsilon>0$, there is an integer
$p_0=p_0(\varepsilon)>0$ such that
\[
 \sup_{j\in\F^\times}\Lambda_{S_j}(\vartheta)
 \le\Lambda_\sharp(\vartheta)+\varepsilon,
 \qquad 0\le\vartheta\le4,
\]
for every prime $p\ge p_0$, where
\[
\Lambda_\sharp(\vartheta):=
\begin{cases}
0,&0\le\vartheta\le\dfrac32,\\[1mm]
\dfrac{2\vartheta-3}{4},&\dfrac32\le\vartheta\le2,\\[2mm]
\dfrac{\vartheta-1}{4},&2\le\vartheta\le3,\\[2mm]
\dfrac12,&3\le\vartheta\le4.
\end{cases}
\]
\end{theorem}
The standard Fourier decay estimate for spheres
(see \cite[Lemma~9]{IK08}), together with Plancherel's theorem,
implies that, for every $\varepsilon>0$ and all sufficiently
large primes $p$,
\[
\sup_{j\in\F^\times}\Lambda_{S_j}(\vartheta)
\leq
\begin{cases}
0,
&0\leq\vartheta\leq\dfrac32,\\[1mm]
\dfrac{2\vartheta-3}{4},
&\dfrac32\leq\vartheta\leq\dfrac52,\\[2mm]
\dfrac12,
&\dfrac52\leq\vartheta\leq4
\end{cases}
\quad +\varepsilon.
\]
In particular, this gives the upper bound
$\sup_{j\in\F^\times}\Lambda_{S_j}(2)\leq1/4+\varepsilon$
at the critical support scale $\vartheta=2$. Theorem~\ref{thm_main_Lambda} improves the displayed bound
throughout the interval \(2<\vartheta<3\). At $\vartheta=2$, the bound remains $1/4+\varepsilon$, whereas Conjecture~\ref{conj_main} asks for a quantity tending to zero. Thus, these estimates do not settle the localized conjecture.

Proposition~\ref{thm:restriction} gives a general transfer principle
from localized bounds to global extension estimates. Applied to
Conjecture~\ref{conj_main}, it gives the range $r>3$.
Theorem~\ref{thm:main-endpoint-three} already proves the stronger
unconditional range $r\ge3$. The localized conjecture asks for the
additional dependence on the size of the physical support.

Localized estimates also give bounds for pinned distances.

For \(E,F\subseteq\F_q^n\) and \(y\in \F_q^{n}\), define
\[
  \Delta_y(E):=\{Q_n(x-y):x\in E\},
  \qquad
  \Pin(E;F):=\max_{y\in F}|\Delta_y(E)|.
\]
Also define
\[
\Delta(E,F):=\{Q_n(x-y):\, x\in E, y\in F\},\qquad \Delta(E):=\Delta(E,E).
\]
The Erd\H{o}s--Falconer distance conjecture predicts the exponent $n/2$ in even dimensions: for every $\varepsilon>0$, sets $E\subseteq\mathbb F_q^n$ with $|E|\ge q^{n/2+\varepsilon}$ should satisfy $|\Delta(E)|\gtrsim_\varepsilon q$. Iosevich and Rudnev \cite{IosevichRudnev2007} formulated the conjecture in 2005 and proved that $|E|\gtrsim q^{\frac{n+1}{2}}$ implies $|\Delta(E)|\gtrsim q$ in every dimension. This bound is sharp in odd dimensions \cite{hart}. The small power loss in the even-dimensional conjecture cannot in general be replaced by a constant factor. In four dimensions, cylinders
over integer multiplication tables give sets of size at least $Cp^2$ with
$o(p)$ distances for every fixed $C$, by the multiplication-table
estimate of Ford~\cite[Corollary~3]{Ford2008}. We discuss this obstruction
in Section~9.

Despite considerable effort \cite{chapman, cheong, JF, KSun, mathz, KohPhamVinh2021, koh2023, BL, mu2, pham, pham11, pham2020, Pham-Yoo, PX}, the conjecture in even dimensions remains open after more than twenty years. In the plane, the best known exponents are $\frac{5}{4}$ over prime fields \cite{mu2} and $\frac{4}{3}$ over arbitrary finite fields \cite{chapman}. Pham, Shen and Xue~\cite[Theorem~1.1]{PSX} reduce the problem in higher even dimensions to the planar problem for the relevant binary quadratic forms. Their Theorem~1.2 gives the pinned exponent $\frac{n}{2}+\frac{1}{4}$ over prime fields: some pin determines a positive proportion of all distances. Their Corollary~1.3 gives the unpinned exponent $\frac{n}{2}+\frac{1}{3}$ over arbitrary finite fields of odd characteristic. Zhang~\cite{ZhangLP} also proved this unpinned exponent by a linear programming argument.

For two sets in $\F^4$, we consider the following conjecture.

\begin{conjecture} \label{conj_pin_distance}
Let $\varepsilon>0$. For any subsets \(E,F\subseteq\F^4\) with\footnote{ Since $\Delta(E,F)=\Delta(F,E)$, the ordering $|F|\leq|E|$ entails no loss of generality in this unpinned statement.} $|F|\le |E|$ and $|E||F|\geq p^{4+\varepsilon}$, we have
\[
|\Delta(E, F)|=(1+o_\varepsilon(1))p.
\]
\end{conjecture}

The two-set conjecture also includes pairs of sets of very different sizes.

Here and throughout the paper, the term \(o_\varepsilon(1)\) denotes a quantity that may depend on \(\varepsilon\) and \(p\), and tends to zero as \(p\to\infty\) for each fixed \(\varepsilon\). In statements quantified over sets \(E\) and \(F\), the convergence is uniform over all sets satisfying the stated hypotheses.

The following theorem uses an upper bound for $\Lambda_{S_j}$ to estimate distances from almost every pin.

\begin{theorem}\label{thm:almost-every-asymmetric-pinned}
Suppose that $p_3>2$ is an integer such that
\[
\sup_{j\in \F^\times}\Lambda_{S_j}(\vartheta)\leq \Lambda(\vartheta),\qquad 0\leq \vartheta\leq 4
\]
for every prime $p\geq p_3$. Let $\varepsilon>0$, let
$p\geq p_3$ be prime, and let \(E,F\subseteq\F^4\). If
$1\leq |F|\leq p^\vartheta$ and
$|E||F|\geq p^{4+2\Lambda(\vartheta)+\varepsilon}$ for some
$0\leq \vartheta\leq 4$, then there exists \(F'\subseteq F\)
such that $|F'|=(1+o_{\varepsilon,p_3}(1))|F|$ and
\[
|\Delta_y(E)|=(1+o_{\varepsilon,p_3}(1))p
\]
uniformly for every \(y\in F'\).
\end{theorem}

Applying this theorem with $\Lambda_\flat$ gives the following stronger form of Conjecture~\ref{conj_pin_distance}.

\begin{corollary} \label{cor_distance_1}
Assume Conjecture \ref{conj_main} holds and let $\varepsilon>0$. Then, for any subsets \(E,F\subseteq\F^4\) with $|F|\le |E|$ and $|E||F|\geq p^{4+\varepsilon}$, there exists \(F'\subseteq F\) such that $|F'|=(1+o_\varepsilon(1))|F|$ and $|\Delta_y(E)|=(1+o_\varepsilon(1))p$ uniformly for every \(y\in F'\).
\end{corollary}

This corollary explains why we study localized restriction estimates. Applying the standard distance argument to the full spherical restriction conjecture in an even dimension $d$ gives the sufficient exponent
\[
\frac{d}{2}+\frac{d}{2d+2},
\]
rather than $d/2$ (see, for example, \cite{KohPhamVinh2021}). In dimension four, these exponents are $12/5$ and $2$. The localized estimate proved below improves $12/5$ to $7/3$. Conjecture~\ref{conj_main} would give the expected exponent $2$, with a $p^\varepsilon$ loss in the hypothesis, and a conclusion for almost every pin.

The global endpoint and Theorem~\ref{thm:almost-every-asymmetric-pinned} imply the following result, with no restriction on the size of $F$.

\begin{corollary} \label{cor_distance_2}
Let \(\varepsilon>0\), and let \(E,F\subseteq\F^4\). If
$|E||F|^{2/3}
\geq p^{4+\varepsilon}$, then there exists \(F'\subseteq F\) such that
$|F'|=(1+o_\varepsilon(1))|F|$ and
$|\Delta_y(E)|=(1+o_\varepsilon(1))p$ uniformly for every \(y\in F'\).
\end{corollary}

The localized bound improves this hypothesis when $p^2\le|F|\le p^3$.

\begin{corollary}\label{cor:improved-middle-pins}
Let $\varepsilon>0$, and let $E,F\subseteq\F^4$ satisfy
\[
 p^2\le|F|\le p^3,
 \qquad |E||F|^{1/2}\ge p^{7/2+\varepsilon}.
\]
Then there exists $F'\subseteq F$ such that
$|F'|=(1+o_\varepsilon(1))|F|$ and
$|\Delta_y(E)|=(1+o_\varepsilon(1))p$ uniformly for every $y\in F'$.
\end{corollary}

\begin{corollary}\label{cor:same-set-seven-thirds}
Let $\varepsilon>0$, and let $E\subseteq\F^4$ satisfy
$|E|\ge p^{7/3+\varepsilon}$. Then there exists $E'\subseteq E$
such that $|E'|=(1+o_\varepsilon(1))|E|$ and
$|\Delta_y(E)|=(1+o_\varepsilon(1))p$ uniformly for every $y\in E'$.
\end{corollary}

The same conclusion holds whenever $|E|/p^{7/3}\to\infty$, with an
error tending to zero as this ratio and $p$ tend to infinity.
For comparison, \cite[Theorem~1.2]{PSX} gives at least one pin
determining a positive proportion of all distances when
$|E|\ge Cp^{9/4}$. Corollary~\ref{cor:same-set-seven-thirds} requires
a larger set and gives a stronger conclusion: almost every pin
determines asymptotically all distances.
Corollary~\ref{cor_distance_concrete} gives the conclusions for all four
ranges in the definition of $\Lambda_\sharp$. In each case, almost
every pin determines asymptotically all distances, rather than only
a positive proportion of them.

We are not aware of an earlier finite field result with the same
asymmetric hypothesis and conclusion for almost every pin. Theorem~1.3
of Iosevich and Liu~\cite{IL} gives a related Euclidean result and
suggests the following finite field bound in four dimensions:

For every \(\varepsilon>0\), the condition
\[
    |E||F|^{3/5}\geq p^{4+\varepsilon}
\]
should guarantee the existence of many \(y\in F\) such that
\[
    |\Delta_y(E)|=(1+o_\varepsilon(1))p.
\]
Corollary~\ref{cor_distance_2} gives the same conclusion under the weaker
condition
$|E||F|^{2/3}
\geq p^{4+\varepsilon}$.

\subsection{Main ideas}\label{subsec:proof-overview}
For the paraboloid and the cone, the relevant Fourier kernels are
given by quadratic Gauss sums, which allow geometric methods to be
used in restriction estimates. For spheres of nonzero radius, the
kernel involves Kloosterman-type sums. The pointwise Weil bound gives
square-root cancellation, but is not enough to estimate the operators
arising from interactions between translated spheres. We instead
bound these operators directly.

Earlier work treated several structured classes of test functions.
Kang and Koh~\cite{Kohkang} proved the conjectured estimates for
certain classes, including $d$-coordinate functions and homogeneous
functions of degree zero. Kang and Koh~\cite{KK2025} introduced an
$\mathcal S$-operator framework connecting spherical restriction
to a dimension-changing operator, and Kang, Koh and Yang~\cite{KKY}
studied its mapping properties. Our argument applies to arbitrary
complex functions.

\paragraph{The operator reduction.}
The standard form $Q_4$ is split over every odd prime field and is
linearly equivalent to the determinant on $2\times2$ matrices (see Lemma~\ref{lem:split-determinant-coordinates}).
Its nonzero levels can therefore be identified with
$\mathrm{SL}_2(\F)$ after an invertible linear change of coordinates.
We change the Fourier pairing accordingly. This preserves counting
norms and extension estimates.

Recall that $|S_j|=p^3-p$. For each $t\in\F^4$, define a matrix
indexed by $x,y\in S_j$ by
\[
 M_t(x,y)=\1_{S_j}(x+y-t)-\frac{|S_j|}{p^4}.
\]
An identity for product vectors reduces a mixed convolution estimate
to quadratic forms associated with these matrices. When $Q_4(t)\ne j$,
we prove
\[
 \langle M_tu,u\rangle\lesssim p\|u\|_2^2.
\]
The case $Q_4(t)=j$ is treated separately. The reduction applies to
arbitrary complex functions.

\paragraph{The arithmetic operator estimate.}
For $Q_4(t)\ne j$, the main operators have kernels of the form
\[
 K_\lambda(x,y)=
 \sum_{z^2=x^{\mathsf T}Dx+y^{\mathsf T}Dy-2x^{\mathsf T}Ry}
       \chi(\lambda z),\qquad x,y\in\F^2,
\]
where $D,R$ are explicit symmetric matrices. Let $J$ be the all-ones
matrix. A Gaussian calculation gives
\[
 K_\lambda-\1_{\{\lambda=0\}}J
 =G_\eta\,\eta(\det R)
   \sum_{t\ne0}\eta(t)\chi(\lambda^2t/4)\widetilde\rho(h_t),
 \qquad |G_\eta|=\sqrt p,
\]
where $\eta$ is the quadratic character, $\widetilde\rho$ is a finite
Weil representation, and
\[
 h_t=\begin{pmatrix}a&t\\(a^2-1)/t&a\end{pmatrix}
       \in\mathrm{SL}_2(\mathcal K)
\]
for a two-dimensional quadratic algebra $\mathcal K$. The parameters
are specified in the operator reduction.

The weighted representation sum is a compression of a degree-one
automorphic Hecke operator over $\F_p(X)$. The level characters exclude
the determinant residual spectrum. Drinfeld's theorem~\cite{Drinfeld}
and the continuous spectral decomposition~\cite{Flicker} give a norm
at most $2\sqrt p$. We identify the complete modification
correspondence, its coefficient transports, and the equal Petersson
masses of the two bundle fibers. Consequently, the estimate holds on
the entire coefficient space without a dimension factor. Together
with the Gaussian factor, it gives an operator norm at most $2p$
for the centered root matrix. The finite Weil representation is used
with its exact normalization~\cite{GH}.

The quadratic algebra may be split, a quadratic field, or nonreduced.
We treat all three cases. Separate affine estimates handle the
remaining rank-one operators for both square classes of their
parameters. The central translation in the case $Q_4(t)\ne j$ is
treated using the regular semisimple entries of the finite
$\mathrm{SL}_2$ character table in~\cite{BFL}. Together, these estimates
give the required bound when $Q_4(t)\ne j$.

\paragraph{Exceptional translations and the refined fourth moment.}
For $Q_4(t)=j$, the non-null part embeds in projective space, where
the centered polarity matrix has norm $p$. For the null part, we need
a mixed estimate between a sphere of nonzero radius and the zero
cone. We project onto characters on isotropic lines and bound the
cross-Gram matrices of these projections. This gives a uniform bound
for the resulting fourth-power sum, for arbitrary complex coefficients.

Combining the two cases and applying a positive operator compression
gives the refined fourth-moment estimate
\[
 \|f*f\|_2^2
 \lesssim p\|f\|_2^4
       +p^{-1}\|Tf\|_\infty^2\|f\|_2^2,
 \qquad \operatorname{supp}f\subseteq S_j,
\]
where $Tf(x)=\sum_{\xi\in S_j}f(\xi)\chi(x\cdot\xi)$ and the norms
in this display use counting measure. The term involving
$\|Tf\|_\infty$ accounts for large Fourier coefficients and cannot in
general be omitted.

\paragraph{The orthogonal decomposition.}
Normalize $\|f\|_2=1$. To separate the large Fourier coefficients,
we use the unit vectors
\[
 a_x(\xi)=|S_j|^{-1/2}\chi(-x\cdot\xi),\qquad \xi\in S_j.
\]
At each step we project onto the span of all previously selected
vectors. The remainder stays orthogonal to this span, so the
coefficients already removed remain zero. The decrease in the squared
$L^2$ norm allows at most $O(p)$ selections. Fourier decay of the
sphere then gives a uniform bound for the inverse of their Gram
matrix. We obtain
\[
 f=h+r,\qquad h=\sum_{i=1}^L c_i a_{x_i},\qquad
 L\lesssim p,\quad\sum_i|c_i|^2\lesssim1,
 \quad\|Tr\|_\infty\lesssim p.
\]
With $\mathcal F=p^{-2}T$, the refined fourth moment gives
$\|\mathcal Fr\|_4^4\lesssim p^{-3}$. Interpolation with Plancherel
gives $\|\mathcal Fr\|_3^3\lesssim p^{-3/2}$.

For $h$, write
\[
 Th=b+e,\qquad
 b=|S_j|^{1/2}\sum_i c_i\1_{\{x_i\}}.
\]
The points $x_i$ are distinct, so $\|b\|_3^3\lesssim p^{9/2}$.
Fourier decay away from these points gives $\|e\|_\infty\lesssim\sqrt p$,
and Plancherel gives $\|e\|_2\lesssim p^2$. Hence,
$\|\mathcal Fh\|_3^3\lesssim p^{-3/2}$. Rescaling proves
Theorem~\ref{thm:main-endpoint-three} for arbitrary complex functions,
without a dyadic summation or a $p^\varepsilon$ loss.

\paragraph{Localization from the same decomposition.}
For a set $A\subseteq\F^4$, the preceding estimates give
\[
 \|\1_A\mathcal Ff\|_2^2
 \lesssim p^{-4}\|b\|_2^2
       +p^{-4}|A|\|e\|_\infty^2
       +|A|^{1/2}\|\mathcal Fr\|_4^2
 \lesssim p^{-1}+|A|p^{-3}+|A|^{1/2}p^{-3/2}.
\]
Converting back to normalized surface measure proves
\[
\|(f\,d\sigma_j)^\vee\|_{L^2(A)}^2
 \le C\left(1+\frac{|A|}{p^2}+\sqrt{\frac{|A|}{p}}\right)\|f\|_{L^2(S_j,d\sigma_j)}^2.
\]
Combining this bound with
Fourier decay and Plancherel gives the four branches of
$\Lambda_\sharp$. Applying H\"older's inequality to the global $L^3$
estimate alone would give only
$\Lambda_{S_j}(\vartheta)\le\vartheta/6+o(1)$.

The rest of the paper is organized as follows. Section~2 records
the preliminary lemmas. Section~3 gives the operator reduction,
Section~4 proves the arithmetic operator estimates, and Section~5
establishes the refined fourth moment, the decomposition, and the
strong endpoint. Section~6 proves the localized estimates and studies
the case $\vartheta=2$. Section~7 derives global estimates from
localized estimates, and Section~8 proves the pinned-distance
results. Section~9 discusses the remaining questions.

\subsection{Some remarks}
\paragraph{Extensions over arbitrary finite fields.}
In this version, we state and prove our main results over odd prime fields. The arithmetic operator method also yields the sharp estimate \(R_{S_j}^*(2\to3)\lesssim1\) for nonzero spheres \(S_j\subseteq\mathbb F_q^4\), uniformly over all odd prime powers \(q\) and all nonzero radii \(j\). The argument extends to characteristic greater than three with the same operator estimates, while characteristic three requires additional arguments for the rank-one and central translation matrices. We will include the complete proof over arbitrary finite fields of odd characteristic in the next version.

\paragraph{Relation to weighted restriction in Euclidean space.}
Conjecture~\ref{conj_main} asks how much of the $L^2$ mass of a spherical
extension can concentrate on a physical set of prescribed cardinality.
The frequency measure remains the normalized surface measure on $S_j$,
while the physical set is arbitrary. At the critical scale, the
conjecture is equivalent to the estimate
\[
 \sum_{x\in E}|(f\,d\sigma_j)^\vee(x)|^2
 \lesssim_\varepsilon p^\varepsilon
 \|f\|_{L^2(S_j,d\sigma_j)}^2,
 \qquad |E|\leq p^2,
\]
for every $\varepsilon>0$, uniformly over odd primes $p$, nonzero radii
$j$, all such sets $E$, and all functions $f$ on $S_j$. Thus, the
bound must hold for every support of the prescribed size, without
geometric or arithmetic assumptions.

Taking the weight to be $\1_E$ relates this question to weighted
$L^2$ restriction. The Mizohata--Takeuchi problem asks for a bound on
the weighted extension norm in terms of the largest mass of the weight
along a line, or inside a tube in a local formulation (see, for example,
\cite[Conjecture~1.2]{CairoZhangMT}). Conjecture~\ref{conj_main} instead
asks for a bound in terms of the support size alone. It is therefore
not a direct finite field formulation of the Mizohata--Takeuchi
conjecture, although both concern weighted $L^2$ estimates.

A closer comparison is the work of Iosevich and
Zhang~\cite{IosevichZhangWeighted}. They prove different localized $L^2$
extension estimates for the sphere and the paraboloid on the same
supports, with stronger estimates for the sphere on suitable truncated
anisotropic lattices. They also discuss general separated supports
and finite field phenomena. In our problem, $\Lambda_{S_j}(\vartheta)$
is defined using every support of the permitted size.
Conjecture~\ref{conj_main} asks for a sharp bound uniformly over all
such sets, including those with affine or arithmetic concentration.

The refined fourth moment and orthogonal decomposition give
unconditional estimates for all these supports. Their dependence on
$|E|$ gives the restriction and distance consequences above, but does
not give the conjectured bound at $\vartheta=2$.
\paragraph{Applications of restriction and extension estimates.}

Restriction and extension estimates have applications to a range
of problems in discrete geometry, including incidence problems
\cite{KP.22, KLP22}, distance questions
\cite{chapman, jchang, KohPhamVinh2021}, and projection problems
\cite{HHHKP}. These applications provide further motivation for
studying restriction and extension estimates.

\section{Fourier and arithmetic preliminaries}\label{sec:setup}

We first define the Fourier transforms used in this paper. We then record the sphere kernel and a norm bound for Hecke operators.

\subsection{Fourier conventions and split coordinates}

Throughout, $p$ is an odd prime and $\F=\F_p$. Let $\chi:\F\to\C^\times$ be the canonical nontrivial additive character, i.e., $\chi(t)=e^{2\pi i t/p}$ $(t\in \F)$. Unless a measure is specified, $L^r$ norms
on finite sets use counting measure. Convolutions also use counting
measure. Our inner product is linear in the first variable:
\[
 \langle f,g\rangle=\sum_x f(x)\overline{g(x)},
 \qquad (f*g)(z)=\sum_x f(x)g(z-x).
\]
For $f:\F^n\to\C$, the unnormalized Fourier transform and its inverse are
\[
 \widehat f(\xi)=\sum_x f(x)\chi(-x\cdot\xi),\qquad
 f(x)=p^{-n}\sum_\xi\widehat f(\xi)\chi(x\cdot\xi).
\]
In particular, $\|\widehat f\|_2=p^{n/2}\|f\|_2$.

Let $Q$ be a split nondegenerate quadratic form $Q$ on $\F^4$. Its associated bilinear form is
\[
B(x,y)=\frac{Q(x+y)-Q(x)-Q(y)}2.
\]
For example, when $Q=Q_4$, one has $B(x,y)=x\cdot y$.

For $j\in \F^\times$, define
\begin{equation}\label{eq:sphere-size-density-notation}
S_j=\{\xi:Q_4(\xi)=j\},\qquad \sigma_j=\frac1{|S_j|}\1_{S_j}.
\end{equation}
The extension operator is
\[
 (f\,d\sigma_j)^\vee(x)=\frac1{|S_j|}\sum_{\xi\in S_j} f(\xi)\chi(B(x,\xi)),\qquad (x\in \F^4).
\]
We also use the transforms
\begin{equation}\label{eq:unitary-transform-convention}
 Tf(x)=\sum_{\xi\in\F^4}f(\xi)\chi(B(x,\xi)),
 \qquad \mathcal Ff=p^{-2}Tf.
\end{equation}
Since $B$ is nondegenerate, $\mathcal F$ is unitary. Moreover,
\begin{equation}\label{eq:unitary-convolution-convention}
 \mathcal F(f*g)=p^2(\mathcal Ff)(\mathcal Fg),\qquad
 (f\,d\sigma_j)^\vee=\frac{p^2}{|S_j|}\mathcal Ff,
 \qquad \|f\|_{L^2(S_j,d\sigma_j)}=|S_j|^{-1/2}\|f\|_2.
\end{equation}
When $Q=Q_4$, we have $Tf(x)=\widehat f(-x)$ and
$\mathcal Ff(x)=p^{-2}\widehat f(-x)$.

\begin{lemma} \label{lem:split-determinant-coordinates}
The form $Q_4$ is linearly equivalent to the determinant form on
$M_2(\F)$ for every odd prime $p$. For each $j\ne0$, there is an
invertible linear map $L_j:\F^4\to M_2(\F)$ such that
\[
 \det L_j(x)=j^{-1}Q_4(x).
\]
This map identifies $S_j$ with $\mathrm{SL}_2(\F)$ and preserves
counting norms and additive convolution. An invertible change of
physical coordinates preserves the Fourier pairing.
\end{lemma}

\begin{proof}
Choose $a,b\in\F$ with $a^2+b^2=-1$. Such a choice exists because the two sets $\{a^2:a\in\F\}$ and $\{-1-b^2:b\in\F\}$ each have $(p+1)/2$ elements. Put
\[
 L(x)=
 \begin{pmatrix}
 x_1+a x_3+b x_4&x_2+b x_3-a x_4\\
 -x_2+b x_3-a x_4&x_1-a x_3-b x_4
 \end{pmatrix}.
\]
Direct expansion gives $\det L(x)=Q_4(x)$. The map is invertible:
$x_1,x_2$ are recovered from the sums and differences of the entries,
and the coefficient matrix for $(x_3,x_4)$ has determinant
$-a^2-b^2=1$. Set
$L_j(x)=\operatorname{diag}(j^{-1},1)L(x)$.

An invertible linear map preserves counting measure and takes sums to
sums, so it preserves every counting norm and additive convolution.
If $B_{\det}$ is the polarized determinant form, polarization gives
\[
 B_{\det}(L_jx,L_j\xi)=j^{-1}x\cdot\xi,
 \qquad B_{\det}(jL_jx,L_j\xi)=x\cdot\xi.
\]
Thus, the frequency change $\xi\mapsto L_j\xi$ is accompanied by the
physical change $x\mapsto jL_jx$, which is again a counting-measure
bijection. The same argument applies to any split nondegenerate
four-dimensional form after choosing a hyperbolic basis.
\end{proof}

This change of coordinates works for every odd prime $p$ and every
$j\ne0$. In particular, it does not require $-1$ or $j$ to be a square.
The arithmetic calculations below are made for $p>3$. Direct bounds
for the finite matrices will cover $p=3$.

\subsection{Gaussian sums and the sphere kernel}

Let $\eta$ be the Legendre symbol, i.e., the non-trivial quadratic character of $\F^\times$, extended by $\eta(0)=0$, and put $G_\eta=\sum_{u\in\F}\chi(u^2)$. The following lemma collects the basic properties of the Gauss sums.

\begin{lemma}\label{lem:one-dimensional-gauss-sum}
One has $|G_\eta|=p^{1/2}$ and $G_\eta^2=\eta(-1)p$. For $s\ne0$,
\[
 \sum_{u\in\F}\chi(su^2+vu)
   =G_\eta\eta(s)\chi\left(-\frac{v^2}{4s}\right).
\]
\end{lemma}

\begin{lemma}\label{lem:quadratic-gauss-transform}
Let $A$ be an invertible symmetric $k$-by-$k$ matrix over $\F$.
For $s\ne0$ and $v\in\F^k$,
\[
 \sum_{u\in\F^k}\chi(su^{\mathsf T}Au+v^{\mathsf T}u)
 =G_\eta^k\eta(s)^k\eta(\det A)
       \chi\left(-\frac{v^{\mathsf T}A^{-1}v}{4s}\right).
\]
In particular, for a split nondegenerate form $Q$ in four variables,
\[
 \sum_{\xi\in\F^4}\chi(sQ(\xi)+B(x,\xi))
 =p^2\chi\left(-\frac{Q(x)}{4s}\right).
\]
\end{lemma}

\begin{proof}
Complete the square and diagonalize $A$ by congruence. Applying
Lemma~\ref{lem:one-dimensional-gauss-sum} to the resulting coordinates
proves the first formula. For a split form in four variables, a
hyperbolic basis shows that the determinant is a square. Together
with $G_\eta^4=p^2$, this gives the second formula.
\end{proof}

We use the classical twisted Kloosterman bound in the following form (see \cite{Weil1948}).

\begin{lemma}\label{lem:twisted-kloosterman}
Let $\omega$ be a multiplicative character of $\F^\times$. If
$(A,B)\ne(0,0)$, then
\[
 \left|\sum_{s\ne0}\omega(s)\chi(As+B/s)\right|\le2\sqrt p.
\]
When $A=B=0$, the sum is $p-1$ if $\omega=1$, and zero otherwise.
\end{lemma}

\begin{lemma}\label{lem:size-kernel}
For a split nondegenerate $Q$ in four variables and $j\ne0$,
\begin{equation}\label{eq:unnormalized-sphere-kernel}
 \sum_{\xi\in S_j}\chi(B(x,\xi))
 =p^3\1_{\{0\}}(x)
   +p\sum_{s\ne0}\chi\left(-js-\frac{Q(x)}{4s}\right).
\end{equation}
Consequently, $|S_j|=p^3-p$ and
\begin{equation}\label{eq:normalized-sphere-kernel}
 (d\sigma_j)^\vee(x)
 =\frac{p^2}{p^2-1}\1_{\{0\}}(x)+\widetilde K_j(x),
 \qquad
 \widetilde K_j(x)=\frac1{p^2-1}
       \sum_{s\ne0}\chi\left(-js-\frac{Q(x)}{4s}\right).
\end{equation}
\end{lemma}

\begin{proof}
Insert
$\1_{S_j}(\xi)=p^{-1}\sum_s\chi(s(Q(\xi)-j))$.
The $s=0$ term is $p^3\1_{\{0\}}(x)$, and
Lemma~\ref{lem:quadratic-gauss-transform} evaluates each remaining term.
At $x=0$, the nonzero-$s$ sum is $-1$, proving the cardinality formula.
Division by $|S_j|$ gives the normalized identity.
\end{proof}

\begin{lemma}\label{lem:kernel-linfty}
Uniformly in $p$ and $j\ne0$,
\[
 \|\widetilde K_j\|_\infty\lesssim p^{-3/2},\qquad
 \left|\sum_{\xi\in S_j}\chi(B(x,\xi))\right|
      \le2p^{3/2}\quad(x\ne0).
\]
\end{lemma}

\begin{proof}
If $Q(x)=0$, the sum over $s\ne0$ in
\eqref{eq:unnormalized-sphere-kernel} is $-1$.
Otherwise, Lemma~\ref{lem:twisted-kloosterman} bounds its modulus by
$2\sqrt p$. The two assertions follow from
\eqref{eq:unnormalized-sphere-kernel} and
\eqref{eq:normalized-sphere-kernel}.
\end{proof}

\begin{lemma} \label{lem:global-auxiliary-bounds}
For every $f:\F^4\to\C$ and every $j\ne0$,
\[
 \|\widehat f\|_{L^2(S_j,d\sigma_j)}
 \lesssim \|f\|_2+p^{-3/4}\|f\|_1,
 \qquad
 \|\widehat f\|_{L^2(S_j,d\sigma_j)}
 \le\frac{p^2}{\sqrt{|S_j|}}\|f\|_2\lesssim p^{1/2}\|f\|_2.
\]
The same estimates hold with the Fourier pairing $B$ for any split $Q$.
\end{lemma}

\begin{proof}
Expanding the restriction square and using
\eqref{eq:normalized-sphere-kernel} and
Lemma~\ref{lem:kernel-linfty} gives
\[
 \|\widehat f\|_{L^2(S_j,d\sigma_j)}^2
 \le\frac{p^2}{p^2-1}\|f\|_2^2
       +\|\widetilde K_j\|_\infty\|f\|_1^2
 \lesssim\|f\|_2^2+p^{-3/2}\|f\|_1^2.
\]
Taking square roots proves the first estimate. Plancherel gives
\[
 \frac1{|S_j|}\sum_{\xi\in S_j}|\widehat f(\xi)|^2
 \le\frac{p^4}{|S_j|}\|f\|_2^2,
\]
which proves the second.
\end{proof}

The following is the $\varepsilon$-removal lemma.

\begin{lemma} \label{lem:epsilon-removal}
Fix $2\le b<\infty$. Suppose that, for every $\varepsilon>0$,
there is a constant $C_\varepsilon$ such that
\[
 \|(f\,d\sigma_j)^\vee\|_{L^b(\F^4)}
 \le C_\varepsilon p^\varepsilon
       \|f\|_{L^2(S_j,d\sigma_j)}
\]
for every odd prime $p$, every $j\ne0$, and every complex function
$f$ on $S_j$. Then, for every $r>b$, there is a constant $C_r$ such that
\[
 \|(f\,d\sigma_j)^\vee\|_{L^r(\F^4)}
 \le C_r\|f\|_{L^2(S_j,d\sigma_j)}
\]
uniformly in $p$, $j$, and $f$. The constant depends only on $b$, $r$,
the constants in the hypothesis, and the absolute kernel bound.
\end{lemma}

\begin{proof}
Lemma~\ref{lem:kernel-linfty} gives
$|(d\sigma_j)^\vee(x)|\le C p^{-3/2}$ for $x\ne0$, with an absolute
constant $C$. For $p\ge C^2$, this is at most $p^{-1}$. Thus, these
surfaces have a common positive Fourier dimension. The epsilon-removal
lemma of \cite[Lemma~16]{european}, with source exponent two and target
exponent $b$, gives the conclusion for $r=b+(r-b)$. Its constants
are uniform for a family with this fixed decay bound and the uniform
constants $C_\varepsilon$. The remaining primes are covered by
Cauchy--Schwarz:
\[
 \|(f\,d\sigma_j)^\vee\|_{L^r(\F^4)}
 \le p^{4/r}\|f\|_{L^2(S_j,d\sigma_j)}.
\]
\end{proof}

\subsection{A norm bound for Hecke operators}
\label{sec:weighted-hecke-operator-bound}

Let $K=\F_p(X)$, let $\mathbb A_K$ be its adele ring, and write
$G=\mathrm{GL}_2$ over $K$ and $Z$ for its center. A unitary central character
$\omega$ is a continuous character of $Z(K)\backslash Z(\mathbb A_K)$.
The scalar automorphic space $\mathcal H_\omega$ consists of measurable
functions satisfying
\[
 F(\gamma z g)=\omega(z)F(g),\qquad
 \|F\|_{\mathcal H_\omega}^2
 =\int_{Z(\mathbb A_K)G(K)\backslash G(\mathbb A_K)}|F(g)|^2\,dg<\infty.
\]
Here $\gamma\in G(K)$ and $z\in Z(\mathbb A_K)$. The absolute value is
well defined on the central quotient because $\omega$ is unitary.

Let $U=\prod_w U_w$ be a compact open level subgroup and let
$\tau:U\to\mathrm U(V)$ be a continuous finite-dimensional unitary representation.
We impose the central compatibility condition
\begin{equation}\label{eq:typed-central-compatibility}
 \tau(z)=\omega(z)^{-1}I_V\qquad(z\in U\cap Z(\mathbb A_K)).
\end{equation}
The space with coefficient representation $\tau$ is
\[
 \mathcal H_\omega(U,\tau)
   =(\mathcal H_\omega\otimes V)^{R\otimes\tau}(U),
 \qquad (R(k)F)(g)=F(gk).
\]
Equivalently, it consists of $V$-valued functions with the convention
\begin{equation}\label{eq:typed-automorphic-convention}
 \Phi(\gamma z g k)=\omega(z)\tau(k)^{-1}\Phi(g),
 \qquad
 \|\Phi\|^2=\int_{Z(\mathbb A_K)G(K)\backslash G(\mathbb A_K)}
                         \|\Phi(g)\|_V^2\,dg.
\end{equation}
Condition~\eqref{eq:typed-central-compatibility} makes the two
transformation rules agree on central elements of $U$.

\begin{proposition}\label{prop:weighted-hecke-bound}
Let $v$ be a rational place. Suppose that
$U_v=\mathrm{GL}_2(\mathcal O_v)$, that $\tau$ is trivial on $U_v$,
and that $\mathcal H_\omega(U,\tau)$ has zero projection onto every
residual determinant representation. Normalize local Haar measure by
$\operatorname{vol}(U_v)=1$. Then the degree-one Hecke operator
\[
 H_v=\1_{U_v\operatorname{diag}(\varpi_v,1)U_v}
\]
has norm at most $2\sqrt p$ on $\mathcal H_\omega(U,\tau)$.
It is the sum of the $p+1$ right cosets, each with coefficient one.
The bound is independent of the dimension and conductor of $\tau$.
It also holds for every compression by isometric coordinate embeddings.
\end{proposition}

\begin{proof}
Drinfeld's theorem \cite{Drinfeld} implies that every unramified local
component of a cuspidal representation of $\mathrm{GL}_2(\mathbb A_K)$
is tempered. No assumption is made on its ramification at the other
places. Thus, at $v$ its normalized Satake parameters $\alpha,\beta$
have modulus one. The eigenvalue of $H_v$ is
$\sqrt p(\alpha+\beta)$ and has modulus at most $2\sqrt p$.
To pass from the usual finite-order-central-character formulation to
an arbitrary unitary $\omega$, twist by a unitary character of the
idele degree. The restriction of $\omega$ to the compact degree-zero
idele class group has finite order. One can choose the twist $\nu$ so
that $\omega\nu^2$ is trivial on a degree-one generator. The twisted
central character then has finite order. Such a twist is unramified
everywhere and does not change the moduli of the Satake parameters.

The continuous spectrum is a direct integral of normalized parabolic
inductions from unitary characters $\mu_1,\mu_2$ (see \cite[Section~1.2]{Flicker}). When the local induction at $v$ has
$U_v$-fixed vectors, its Hecke eigenvalue is
\[
 \sqrt p\bigl(\mu_{1,v}(\varpi_v)+\mu_{2,v}(\varpi_v)\bigr),
\]
whose modulus is at most $2\sqrt p$. The residual spectrum consists
of determinant-character lines (see
\cite[Section~4, Proposition~4.29]{Flicker}). It is excluded by hypothesis.

The operator $H_v$ is normal. Indeed, its inverse double coset is the
original double coset multiplied by the central scalar
$\varpi_v^{-1}I$. Therefore, $H_v^*$ is a unit scalar multiple of
$H_v$ on $\mathcal H_\omega$. The spectral bounds give the claimed
operator norm on the orthogonal complement
$\mathcal H_{\omega,\mathrm{nr}}$ of the residual spectrum.
Residual spectral projections commute with the right level action.
Hence, the hypothesis places the space with coefficients inside
$\mathcal H_{\omega,\mathrm{nr}}\otimes V$. The operator there is
$H_v\otimes I_V$. Tensoring with $I_V$, passing to the closed invariant
space, and compressing by isometries do not increase its norm.
\end{proof}

At a full flag, a diagonal character $(\omega_0,1)$ with $\omega_0\ne1$
excludes every determinant representation: its two diagonal characters
would have to be equal. A nontrivial character on an element of
determinant one in a residue unipotent subgroup also excludes them.
We use both criteria below. The local coefficient characters determine
the inverse of the automorphic central character through
\eqref{eq:typed-central-compatibility}. When extending these characters
along the degree factor, we may choose a finite-order degree character,
for example the trivial one.

We next specify the coordinate weights. Put
$\overline G=Z(\mathbb A_K)\backslash G(\mathbb A_K)$, and let
$\overline\Gamma$ and $\overline U$ be the images of $G(K)$ and $U$.
For a representative $g_a$ of
$\overline\Gamma\backslash\overline G/\overline U$, define
\[
 \Gamma_a=\overline\Gamma\cap
                  \overline g_a\overline U\overline g_a^{-1},
 \qquad w_a=\frac{\operatorname{vol}(\overline U)}{|\Gamma_a|}.
\]
The group $\Gamma_a$ is finite. The allowed coefficient fiber is
\begin{equation}\label{eq:typed-stabilizer-fiber}
 V_a=\{v\in V:\omega(z)\tau(k)v=v
              \text{ whenever }g_a k=\gamma z g_a\},
\end{equation}
where $\gamma\in G(K)$, $z\in Z(\mathbb A_K)$, and $k\in U$.
The scalar $\omega(z)$ is independent of the chosen rational central
factor, and \eqref{eq:typed-central-compatibility} removes the ambiguity
from $U\cap Z(\mathbb A_K)$. Thus, the fiber condition is well defined.

\begin{lemma} \label{lem:typed-hecke-compression}
Under the hypotheses of Proposition~\ref{prop:weighted-hecke-bound},
evaluation on the representatives $g_a$ identifies the space with
coefficients isometrically with
\[
 \left\{(v_a):v_a\in V_a,\quad\sum_a w_a\|v_a\|_V^2<\infty\right\}.
\]
Let $h_1,\ldots,h_{p+1}$ represent the right $U_v$-cosets of
$U_v\operatorname{diag}(\varpi_v,1)U_v$, embedded at the place $v$.
For finite sets of source and target classes, the matrix in orthonormal
fiber coordinates has entries
\begin{equation}\label{eq:typed-compression-matrix}
 B_{ab}=\sqrt{\frac{w_a}{w_b}}\,
 P_{V_a}\sum_{\substack{i:\,[g_a h_i]=[g_b]}}
       \omega(z_i)\tau(k_i)^{-1}\big|_{V_b},
 \qquad g_a h_i=\gamma_i z_i g_b k_i.
\end{equation}
Here $P_{V_a}$ is orthogonal projection onto $V_a$, and every right
coset is counted once. The matrix has norm at most $2\sqrt p$.
If the source and target weights are equal and their stabilizers act
trivially on $V$, the weight ratio and the fiber projections disappear.
\end{lemma}

\begin{proof}
The image of $g_a\overline U$ in
$\overline\Gamma\backslash\overline G$ has measure
$\operatorname{vol}(\overline U)/|\Gamma_a|$. Its stabilizer is the
finite group $\Gamma_a$, and the transformation law
\eqref{eq:typed-automorphic-convention} makes the norm of a section
constant on this orbit. It also shows that a value at $g_a$ extends
consistently over the orbit exactly when it lies in $V_a$.
This proves the asserted Hilbert-space identification.

The isometric embedding $J_a:V_a\to\mathcal H_\omega(U,\tau)$
has value $w_a^{-1/2}v$ at $g_a$ and vanishes on the other classes.
Since the type is trivial at $v$, right convolution gives
$(H_v\Phi)(g_a)=\sum_i\Phi(g_a h_i)$.
For $g_a h_i=\gamma_i z_i g_b k_i$, the summand is
$\omega(z_i)\tau(k_i)^{-1}\Phi(g_b)$.
Applying $J_a^*$ and $J_b$ gives
\eqref{eq:typed-compression-matrix}. Its norm is bounded by the norm
of $H_v$, as claimed.
\end{proof}

For the bundle descriptions used below, the base curve is
$\mathbb P^1$. Quotienting by the adelic center identifies bundles
that differ by tensoring with a line bundle. A rank-two bundle cannot
be isomorphic to its twist by a line bundle of nonzero degree; a
line bundle of degree zero on $\mathbb P^1$ is trivial. If the level
contains all scalar units, this identifies $\Gamma_a$ with
$\operatorname{Aut}(E_a)/\F_p^\times$. Consequently,
\[
 w_a=\operatorname{vol}(\overline U)
                \frac{p-1}{|\operatorname{Aut}(E_a)|}.
\]
Thus, the familiar reciprocal-automorphism weights differ from the
central-quotient weights by one common positive factor. In particular,
two objects with scalar automorphism group $\F_p^\times$ have equal
weights. Their entire coefficient spaces may be used precisely when
the combined stabilizer action in \eqref{eq:typed-stabilizer-fiber}
is trivial.

\subsection{Finite group representations}

We use the unitary linear Weil representation of
$\mathrm{Sp}_4(\F)$ on $L^2(\F^2)$, as in
\cite[Theorem~1.3.2]{GH}. In
Section~\ref{sec:exact-weil-root-normalization}, we specify its action
on generators and derive its Gaussian kernel. These formulas fix
the character and sign conventions. Its restriction to any symplectic
subgroup is again a unitary linear representation. This also applies
to the subgroups arising from a nonreduced quadratic algebra below.

\begin{lemma}\label{lem:sl2-regular-class-bound}
Let $p>3$, let $\mathcal C$ be a regular semisimple conjugacy class
in $\mathrm{SL}_2(\F)$, and let $\chi_\pi$ be a nontrivial irreducible
character. For $g\in\mathcal C$,
\[
 \frac{|\mathcal C|\,|\chi_\pi(g)|}{\chi_\pi(1)}\le2p.
\]
\end{lemma}

\begin{proof}
A split regular class has size $p(p+1)$, and a nonsplit regular
class has size $p(p-1)$. The character table
\cite[Table~1]{BFL} gives modulus at most two for the principal-series
values on a split class, with degree $p+1$, and for the cuspidal
values on a nonsplit class, with degree $p-1$. The exceptional halves
have modulus one and half these degrees. Opposite-type values vanish.
The Steinberg character has modulus one on these classes and degree
$p$. These bounds prove the assertion in every case.
\end{proof}

\section{A convolution identity and reduction to matrices}
\label{sec:exact-product-reduction}

Let $Q$ be a split nondegenerate quadratic form on $\F^4$,
and let $j\ne0$. Recall that $S_j=\{x:Q(x)=j\}$. For $t\in\F^4$,
define the real symmetric matrix on $L^2(S_j)$ by
\begin{equation}\label{eq:centered-shell-matrix}
 M_t(x,y)=\1_{S_j}(x+y-t)-\frac{|S_j|}{p^4},
 \qquad x,y\in S_j.
\end{equation}
Here we subtract the density of $S_j$ in $\F^4$.
The following identity relates these matrices to convolution.

\begin{lemma}\label{lem:product-block-identity}
Let $f:\F^4\to\C$ be supported on $S_j$, and let
$g:\F^4\to\C$ be arbitrary. Put
\[
 a_t(x)=f(x)\overline{g(x-t)},\qquad x\in S_j.
\]
Then
\begin{equation}\label{eq:product-block-expansion}
 \|\1_{S_j}(f*g)\|_2^2-\frac{|S_j|}{p^4}\|f*g\|_2^2
   =\sum_{t\in\F^4}\langle M_ta_t,a_t\rangle,
\end{equation}
and
\begin{equation}\label{eq:product-vector-mass}
 \sum_{t\in\F^4}\|a_t\|_2^2=\|f\|_2^2\|g\|_2^2.
\end{equation}
\end{lemma}

\begin{proof}
Since $\operatorname{supp}f\subseteq S_j$, expanding the square gives
\[
 \|\1_{S_j}(f*g)\|_2^2
 =\sum_z\1_{S_j}(z)\sum_{x,y\in S_j}
       f(x)g(z-x)\overline{f(y)g(z-y)}.
\]
For fixed $x,y$, make the bijective change of variables $t=x+y-z$.
The summand becomes
\[
 \1_{S_j}(x+y-t)
 f(x)g(y-t)\overline{f(y)g(x-t)}
 =\1_{S_j}(x+y-t)a_t(x)\overline{a_t(y)}.
\]
The same expansion without $\1_{S_j}(z)$ gives $\|f*g\|_2^2$.
Subtracting $|S_j|/p^4$ times that identity proves
\eqref{eq:product-block-expansion}. Since $M_t$ is symmetric, the
resulting quadratic form is $\langle M_ta_t,a_t\rangle$ with our
inner-product convention.
Finally,
\[
 \sum_t\|a_t\|_2^2
 =\sum_{x\in S_j}|f(x)|^2\sum_t|g(x-t)|^2
 =\|f\|_2^2\|g\|_2^2.
\]
\end{proof}

In the next section, we prove (see Theorem~\ref{thm:completed-generic-blocks})
\begin{equation}\label{eq:generic-shell-bound}
 \|M_t\|_{2\to2}\le16p
 \qquad\text{whenever }Q(t)\ne j.
\end{equation}
For $p>3$, the proof gives
the constant $11$. This operator norm bound holds on all of
$L^2(S_j)$.

The estimate \eqref{eq:generic-shell-bound} applies only when
$Q(t)\ne j$. When $t\in S_j$, write $x=t+u$ and $y=t+v$.
Since $x,y,t\in S_j$,
\begin{equation}\label{eq:tangent-shell-polarity}
 Q(u)+2B(t,u)=Q(v)+2B(t,v)=0,
 \qquad Q(x+y-t)=j+2B(u,v).
\end{equation}
For nonzero $u,v$, the condition $B(u,v)=0$ is a projective
orthogonality relation.
In Section~\ref{sec:fourth-moments-sharp-extension}, we use it to
treat the terms with $Q(u),Q(v)\ne0$. A convolution estimate between
the sphere and the cone $\{Q=0\}$ treats the remaining terms, using
the vectors $a_t$ from Lemma~\ref{lem:product-block-identity}.
We then combine these estimates with
\eqref{eq:product-block-expansion}--\eqref{eq:product-vector-mass}
to prove a refined fourth-moment bound and the strong endpoint estimate.

\section{Hecke operators and matrix estimates}
\label{sec:hecke-matrix-estimates}

We prove the matrix estimate \eqref{eq:generic-shell-bound} for $Q(t)\ne j$.
We first express a weighted sum of
finite Weil operators as a compression of a degree-one automorphic
Hecke operator. A Gaussian sum calculation then gives the bound for
the two-dimensional root matrices. We use two first-jet
correspondences to bound the affine rank-one matrices. Finally, we
combine these estimates with bounds for the boundary blocks.
All finite-dimensional norms use counting measure, and each root in
a finite field is counted once. We use the equivariance and coordinate
normalization of Lemma~\ref{lem:typed-hecke-compression}, together with
Proposition~\ref{prop:weighted-hecke-bound}. In particular, the
central character of the automorphic functions is the inverse of the
central character of the coefficient type.

\subsection{A Hecke bound for a weighted operator sum}
\label{sec:arithmetic-whittaker-hecke}

\begin{theorem}\label{thm:arithmetic-whittaker-curve}
Let $p>3$, let $\mathcal K$ be a two-dimensional commutative
$\mathbb F_p$-algebra, and let $a\in\mathcal K$ satisfy
$a^2+a+b=0$ for $b\in\mathbb F_p^\times$. Suppose that
$a$ and $d=a^2-1$ are units and that $1,a^2$ are linearly independent.
For $t\in\mathbb F_p^\times$ put
\[
 h_t=\begin{pmatrix}a&t\\d/t&a\end{pmatrix}.
\]
Let $\rho$ be any finite-dimensional unitary representation of
$\mathrm{SL}_2(\mathcal K)$, and let $c\in\mathbb F_p$.
Then
\[
 \left\|\sum_{t\ne0}\eta(t)\chi(ct)\rho(h_t)\right\|
 \le 2\sqrt p.
\]
Here $\eta$ is the quadratic character. The assertion includes the
split algebra, a quadratic field, and the dual-number algebra.
\end{theorem}

\begin{proof}
We construct the automorphic model and check its coefficient
space and measure. The base curve is $\mathbb P^1_X$.
Since $a^2$ generates $\mathcal K$, the element $a^{-2}$ does also,
and evaluation $X\mapsto a^{-2}$ identifies the degree-two algebra
\[
 \mathbb F_p[X]/(P(X))\simeq\mathcal K,
 \qquad
 P(X)=X^2-\frac{1-2b}{b^2}X+\frac1{b^2}.
\]
The divisor $P$ is disjoint from $0,1,\infty$, because $a$ and
$a^2-1$ are units. It may consist of two rational points, a degree-two
point, or a doubled rational point. In the last case all the level
conditions below mean reduction modulo the square of its local
parameter.

First suppose that $\rho(-I)=\epsilon I$, $\epsilon\in\{1,-1\}$.
The general case is the orthogonal sum of these two invariant spaces.
Choose a unitary character $\vartheta_{\mathcal K}$ of $\mathcal K^\times$ with
$\vartheta_{\mathcal K}(-1)=\epsilon$. Such a character exists because the character
group separates the nonidentity element $-1$. Write
\[
 G^+=\{g\in\mathrm{GL}_2(\mathcal K):\det g\in
                         (\mathcal K^\times)^2\}.
\]
There is a representation of $G^+$ defined by
\[
 r(g)=\vartheta_{\mathcal K}(\delta)\rho(g/\delta),\qquad \delta^2=\det g.
\]
This is independent of the choice of square root: changing
$\delta$ to $-\delta$ changes both factors by $\epsilon$.
It is multiplicative, since scalar matrices commute with every matrix.
For the split algebra the equation $\delta^2=\det g$ can have four
solutions, rather than two. In this case decompose further under the
central involutions of $\mathrm{SL}_2(\mathcal K)$, and choose $\vartheta_{\mathcal K}$
whose restriction to every element $u$ with $u^2=1$ equals the scalar
$\rho(uI)$. A character on this subgroup extends to
$\mathcal K^\times$, so the same formula is well-defined. This also
handles all the algebra types uniformly.

Precomposing $\rho$ with a fixed $\mathrm{GL}_2(\mathcal K)$
conjugation preserves these properties. We now specify that
conjugation. Put
\[
 D_0=\operatorname{diag}(-a/d,1),\qquad
 J_0=\begin{pmatrix}0&1\\d&0\end{pmatrix},\qquad Q_0=J_0D_0.
\]
Use in the definition of $r$ the representation
$g\mapsto\rho(Q_0gQ_0^{-1})$.
At the divisor $P$, take the inverse image of $G^+$ as compact level,
with coefficient representation $r$. Together with the types at $1$
and infinity specified below, this defines the compact coefficient
representation $\tau$. We use
$\Phi(\gamma z gk)=\omega(z)\tau(k)^{-1}\Phi(g)$ throughout.

At $X=1$ take the compact subgroup whose residue lies in
\[
 ZU_-=
 \left\{s\begin{pmatrix}1&0\\u&1\end{pmatrix}:
           s\in\mathbb F_p^\times,\ u\in\mathbb F_p\right\}.
\]
Give it the character
\[
 s\begin{pmatrix}1&0\\u&1\end{pmatrix}
 \longmapsto \mu(s)\chi(-cu),
 \qquad \mu=(\vartheta_{\mathcal K}|_{\mathbb F_p^\times}\eta)^{-1}.
\]
At infinity take an ordinary flag, with $\eta$ applied to the scalar
on its line. These three local coefficient characters have product
one on $\mathbb F_p^\times$. They therefore extend to a finite-order
idele class character: equivalently, they define a character of the
degree-zero ray class group, whose presentation on $\mathbb P^1$ is
the product of the residue unit groups modulo the diagonal constants.
Choose the trivial character on the degree factor, and let $\omega$
be the inverse of this extension. Then
$\tau(z)=\omega(z)^{-1}I$ on the compact center, as required by
Lemma~\ref{lem:typed-hecke-compression}. All remaining places have
maximal compact level.

At infinity the local type on the diagonal torus is $(\eta,1)$.
A determinant representation restricts there as $(\nu,\nu)$ for some
character $\nu$, so cannot contain this type: varying the two diagonal
entries independently would force both $\nu=\eta$ and $\nu=1$.
Thus, the selected automorphic $L^2$ space has no residual
determinant representations, including when $c=0$.
When $c\ne0$, the nontrivial unipotent character at $1$ gives a
second exclusion. Proposition~\ref{prop:weighted-hecke-bound} gives
\[
 \|H_0\|\le2\sqrt p
\]
for the unnormalized degree-one Hecke operator at $X=0$.
The maximal compact subgroup at $0$ has Haar volume one, so this
operator sums its $p+1$ right cosets with weight one. The finite index
of $G^+$ at $P$ changes neither this normalization nor the bound.

We next identify a compression of this operator. The source object is
$E=\mathcal O^2$, with the standard $ZU_-$ reduction at $1$, the standard
level at $P$, and the infinity flag $e_1$. The target object is
$F=\mathcal O(1)\oplus\mathcal O$, with the same standard reductions
at the finite places and infinity flag $e_2$. Both have automorphism
group exactly $\mathbb F_p^\times$. Indeed, the source automorphisms
fixing its infinity flag are upper triangular constant matrices,
whose intersection with $ZU_-$ is the center. A target automorphism
has the form
\[
 \begin{pmatrix}\alpha&\beta+\gamma X\\0&\delta\end{pmatrix}.
\]
In the infinity frame this has value
$\left(\begin{smallmatrix}\alpha&\gamma\\0&\delta\end{smallmatrix}\right)$.
Fixing the infinity flag $e_2$ forces $\gamma=0$. Its reduction at $1$
belongs to $ZU_-$ only when $\beta=0$ and $\alpha=\delta$.
The coefficient action of each central constant is trivial after
multiplying the three local factors, by the choice of $\mu$.
Thus, the full representation space of $\rho$ occurs at both objects.
Their coordinate masses are the same common positive multiple of
$(p-1)^{-1}$, and the normalized embeddings $i_E,i_F$ of
Lemma~\ref{lem:typed-hecke-compression} are isometric.
Here quotienting by the adelic center introduces no relative factor.
Indeed, an isomorphism from either bundle to its tensor product with
a line bundle $L$ forces $2\deg L=0$. On $\mathbb P^1$, such $L$
is trivial, and all our compact levels contain the scalar unit groups.
The degree-zero and degree-one objects also remain distinct modulo
tensoring by line bundles, which changes their degree by an even
integer. This observation applies to the two bundle families used
later in the section as well. No projection to a proper subspace of
the coefficient fiber is needed.

For $z\in\mathbb F_p^\times$ consider
\[
 A_z(X)=\begin{pmatrix}1&z(X-1)\\-1/z&1\end{pmatrix},
 \qquad \det A_z(X)=X.
\]
This is an elementary modification $E\hookrightarrow F$ at $0$.
Its kernel at $0$ is $\langle(z,1)\rangle$. In the target infinity
frame the map has value
\[
 \left.\operatorname{diag}(X^{-1},1)A_z(X)\right|_{\infty}
 =\begin{pmatrix}0&z\\-1/z&1\end{pmatrix}.
\]
These are exactly the modifications between the two specified objects.
For completeness, a general bundle map has the form
\[
 \begin{pmatrix}\alpha+\beta X&\gamma+\delta X\\
                  e&f\end{pmatrix}.
\]
Preservation of the two infinity flags forces $\beta=0$ and $e\ne0$.
Membership of its value at $1$ in $ZU_-$ forces
$f=\alpha\ne0$ and $\gamma=-\delta$.
Its determinant vanishes at $0$ precisely when
$\delta e=-\alpha^2$. Dividing by the common central scalar $\alpha$
then gives exactly $A_z$, with $z=\delta/\alpha\ne0$.
The two other kernel lines at $0$ give targets outside the selected
chart. This proves completeness.

We specify representatives to identify the right-coset action. We use
the following matrix $w_0$ throughout the section:
\[
 w_0=\begin{pmatrix}0&1\\1&0\end{pmatrix},\qquad
 g_E=I,\qquad g_{F,\infty}=\operatorname{diag}(X,1)w_0,
\]
with $g_{F,v}=I$ at every finite place. Let $a_z$ be supported at
$0$, with component $A_z(X)$ there, and put $\gamma_z=A_z$ as a
rational matrix. Then
\[
 g_Fa_z=\gamma_zg_Ek_z,\qquad k_{z,0}=I,\qquad
 k_{z,v}^{-1}=g_{F,v}^{-1}A_zg_{E,v}\quad(v\ne0).
\]
The matrices $k_{z,v}$ belong to the stated compact levels. The
right cosets $a_zU_0$ are distinct, since the image lines of
$A_z(0)$ are $\langle(z,-1)\rangle$. Thus, every listed modification
occurs once in the Hecke action. Its coefficient is
$\tau(k_z)^{-1}$, namely the product of the forward framed maps
$g_{F,v}^{-1}A_zg_{E,v}$.

We now compute these coefficient factors. At $1$,
\[
 A_z(1)=\begin{pmatrix}1&0\\-1/z&1\end{pmatrix}
\]
gives $\chi(c/z)$. At infinity the framed map has reduction
\[
 w_0^{-1}\left.\operatorname{diag}(X^{-1},1)A_z(X)\right|_{\infty}
 =\begin{pmatrix}-1/z&1\\0&z\end{pmatrix},
\]
so its factor on the flagged line is $\eta(-1/z)$. At $P$, its determinant is
$a^{-2}$, so take the fixed square root $\delta=a^{-1}$. Direct
multiplication gives
\[
 D_0\,aA_z(a^{-2})\,D_0^{-1}=h_z,
 \qquad J_0h_zJ_0^{-1}=h_{1/z}.
\]
Thus, the $P$ coefficient is the constant scalar of modulus one
$\vartheta_{\mathcal K}(a^{-1})$ times $\rho(h_{1/z})$.
Lemma~\ref{lem:typed-hecke-compression} therefore gives the exact
identity
\begin{equation}\label{eq:weighted-hecke-exact-compression}
 i_F^*H_0i_E
 =\eta(-1)\vartheta_{\mathcal K}(a^{-1})
   \sum_{t\ne0}\eta(t)\chi(ct)\rho(h_t).
\end{equation}
This holds on each central-involution summand used to choose
$\vartheta_{\mathcal K}$. The scalar has modulus one, so the norm
bound for $H_0$ proves the theorem on each summand and hence on their
orthogonal sum.

\end{proof}

In the next subsection, we compute the Gaussian factor and apply
the theorem to the two-dimensional root matrix.

\subsection{Gaussian normalization of the Weil operators}
\label{sec:exact-weil-root-normalization}

We compute the scalar factor, character, and conjugation that
relate the root matrix to the preceding operator sum. We use the
finite Weil representation of $\mathrm{Sp}_4(\mathbb F_p)$, so the
argument also applies to a nonreduced quadratic algebra.

Let $p>3$, $q\notin\{0,1\}$, and
\[
 D=\operatorname{diag}(1,-4q),\qquad
 R=\begin{pmatrix}1&-2q\\-2q&2q(q+1)\end{pmatrix}.
\]
Put $\mathcal K=\mathbb F_p[D^{-1}R]$, $a=R^{-1}D$,
$b=2/(q-1)$, and $d=a^2-1$. Then
\[
 a^2+a+b=0,\qquad
 \det a=b,\qquad \det(a^2-1)=b(b+2)\ne0.
\]
The matrix $a$ is nonscalar, hence $1,a^2$ are independent.
These assertions include $q=9$, when $\mathcal K$ is the dual-number
algebra.

Every element of $\mathcal K$ is self-adjoint for the form $D$.
Consequently, $\mathrm{SL}_2(\mathcal K)$ acts symplectically on
$\mathbb F_p^2\oplus\mathbb F_p^2$ with pairing
\[
 \omega_D((x,\zeta),(y,\theta))
       =x^{\top}D\theta-\zeta^{\top}Dy.
\]
We fix the normalization of the Weil representation explicitly.
Write $G_\eta=\sum_{z\in\mathbb F_p}\chi(z^2)$.

\begin{lemma}
\label{lem:weil-generator-normalization}
Let $p>3$, and give $\F^2\oplus\F^2$ the symplectic form
\[
 \omega_{\mathrm{sp}}((a,b),(a',b'))
       =a^{\top}b'-b^{\top}a'.
\]
In the Schr\"odinger model associated with the additive character
$\chi$, the canonical linear Weil representation $\rho_{\mathrm W}$
has the following action. For a symmetric matrix $C$ and an invertible
matrix $A$, put
\[
 \ell(C)=\begin{pmatrix}I&0\\C&I\end{pmatrix},\qquad
 m(A)=\begin{pmatrix}A&0\\0&A^{-\top}\end{pmatrix},\qquad
 w=\begin{pmatrix}0&I\\-I&0\end{pmatrix}.
\]
Then
\begin{equation}\label{eq:weil-generator-normalization}
 \begin{aligned}
 \rho_{\mathrm W}(\ell(C))f(x)
   &=\chi(-x^{\top}Cx/2)f(x),\\
 \rho_{\mathrm W}(m(A))f(x)
   &=\eta(\det A)f(A^{-1}x),\\
 \rho_{\mathrm W}(w)f(x)
   &=\frac{\eta(-1)}p\sum_{y\in\F^2}\chi(x^{\top}y)f(y).
 \end{aligned}
\end{equation}
Consequently, if a symplectic matrix has invertible upper-right block
$\beta$, its kernel is
\begin{equation}\label{eq:weil-open-cell-kernel}
 \rho_{\mathrm W}
 \begin{pmatrix}\alpha&\beta\\\gamma&\delta\end{pmatrix}(x,y)
 =\frac{\eta(-\det\beta)}p
 \chi\left(-\frac12\left[
 x^{\top}\delta\beta^{-1}x
 -2y^{\top}\beta^{-1}x
 +y^{\top}\beta^{-1}\alpha y\right]\right).
\end{equation}
\end{lemma}

\begin{proof}
We make the convention in the Schr\"odinger model explicit. The
Heisenberg group has multiplication
\[
 (a,b,z)(a',b',z')=
 \left(a+a',b+b',z+z'+
       \tfrac12(a^{\top}b'-b^{\top}a')\right)
\]
and representation
\[
 (\pi(a,b,z)f)(x)
   =\chi(z+b^{\top}x+a^{\top}b/2)f(x+a).
\]
It is irreducible with central character $\chi$. The canonical linear
representation of \cite[Theorem~1.3.2]{GH} satisfies the Egorov identity
\[
 \rho_{\mathrm W}(g)\pi(a,b,z)\rho_{\mathrm W}(g)^{-1}
       =\pi(g(a,b),z).
\]
Define the unitary operators
\[
 L_Cf(x)=\chi(-x^{\top}Cx/2)f(x),\quad
 D_Af(x)=f(A^{-1}x),\quad
 \mathcal F_+f(x)=p^{-1}\sum_y\chi(x^{\top}y)f(y).
\]
Direct substitution shows that they satisfy this identity for
$\ell(C)$, $m(A)$, and $w$, respectively. Indeed, their conjugations
send $(a,b)$ to $(a,b+Ca)$, $(Aa,A^{-\top}b)$, and $(b,-a)$.
By irreducibility, each differs from the corresponding Weil operator
by a scalar. We determine these scalars from the character formula.

For $g-I$ invertible, \cite[Theorem~2.2.1]{GH} gives, in dimension four,
\[
 \operatorname{tr}\rho_{\mathrm W}(g)
   =\frac{G_\eta^4}{p^2}
      \eta\bigl(\det(\kappa(g)+I)\bigr)
   =\eta\bigl(\det(g-I)\bigr),
 \qquad \kappa(g)=(g+I)(g-I)^{-1}.
\]
The last equality follows from $G_\eta^4=p^2$,
$\kappa(g)+I=2g(g-I)^{-1}$, and $\det g=1$.
Since $\det(w-I)=4$, the trace of $\rho_{\mathrm W}(w)$ is one.
On the other hand,
\[
 \operatorname{tr}\mathcal F_+
    =p^{-1}\sum_{x\in\F^2}\chi(x^{\top}x)
    =G_\eta^2/p=\eta(-1).
\]
Thus, $\rho_{\mathrm W}(w)=\eta(-1)\mathcal F_+$.

If $A-I$ is invertible, then $D_A$ fixes only the coordinate at zero,
so its trace is one. Moreover,
\[
 \det(m(A)-I)=\det(A)^{-1}\det(A-I)^2.
\]
The character formula therefore gives
$\rho_{\mathrm W}(m(A))=\eta(\det A)D_A$.
For arbitrary $A$, choose $c\in\F^\times$ which is neither one nor an
eigenvalue of $A$. There are at most three excluded elements, so this
is possible for $p>3$. Both $cI-I$ and $c^{-1}A-I$ are invertible.
Applying the formula to these two factors of $A$ and using multiplicativity
of the Weil representation proves it for every $A$.

Finally, choose $c\in\F^\times\setminus\{1\}$. For
$g=\ell(C)m(cI)$, the determinant of $g-I$ is
$(c-1)^2(c^{-1}-1)^2$, and hence $\operatorname{tr}\rho_{\mathrm W}(g)=1$.
The operator $L_C D_{cI}$ satisfies the Egorov identity for $g$ and
also has trace one: its only fixed coordinate is zero, where the
phase equals one. Therefore,
$\rho_{\mathrm W}(g)=L_C D_{cI}$. Dividing by
$\rho_{\mathrm W}(m(cI))=D_{cI}$ proves the shear formula, including
when $C$ is singular.

For invertible $\beta$, the symplectic identities imply that
$\delta\beta^{-1}$ and $\beta^{-1}\alpha$ are symmetric and give
\[
 \begin{pmatrix}\alpha&\beta\\\gamma&\delta\end{pmatrix}
 =\ell(\delta\beta^{-1})\,m(\beta)\,w\,
                         \ell(\beta^{-1}\alpha).
\]
Multiplying the operators in \eqref{eq:weil-generator-normalization}
in this order proves \eqref{eq:weil-open-cell-kernel}.
\end{proof}

We restrict $\rho_{\mathrm W}$ to the displayed copy of
$\mathrm{SL}_2(\mathcal K)$, acting on $L^2(\mathbb F_p^2)$.
Its embedding uses the action of $\mathcal K$ on $\mathbb F_p^2$
and the nondegenerate form $D$. It is faithful and symplectic also
for the dual-number algebra, with no trace form on $\mathcal K$ needed.

In the $D$-momentum coordinates define
\[
 g_s=
 \begin{pmatrix}
 a&-a/(2s)\\2s(a^{-1}-a)&a
 \end{pmatrix},\qquad s\ne0.
\]
Replacing $\zeta$ by the standard momentum $D\zeta$ gives the
blocks $\alpha=a$, $\beta=-R^{-1}/(2s)$,
$\gamma=2sD(a^{-1}-a)$, and $\delta=DaD^{-1}$. Thus,
\[
 \delta\beta^{-1}=-2sD,\qquad
 \beta^{-1}\alpha=-2sD,\qquad
 \eta(-\det\beta)=\eta(-\det R).
\]
Thus, with $H(x,y)=x^{\top}Dx+y^{\top}Dy-2x^{\top}Ry$,
\begin{equation}\label{eq:exact-binary-weil-kernel}
 \rho_{\mathrm W}(g_s)(x,y)
   =\frac{\eta(-\det R)}p\,\chi(sH(x,y)).
\end{equation}
In particular, the scalar factor is independent of $s$.

Set
\[
 E=\operatorname{diag}(-a/2,1),\qquad
 h_t=\begin{pmatrix}a&t\\d/t&a\end{pmatrix},\qquad
 \widetilde\rho(h)=\rho_{\mathrm W}(EhE^{-1}).
\]
Conjugation by $E$ is an automorphism of
$\mathrm{SL}_2(\mathcal K)$, so $\widetilde\rho$ is a unitary
representation. This only uses conjugation by $E$ as a group
automorphism, not a Weil operator associated with $E$. Direct multiplication gives
$g_s=Eh_{1/s}E^{-1}$.

Define the root matrix by
\[
 K_\lambda(x,y)=\sum_{z^2=H(x,y)}\chi(\lambda z),
 \qquad x,y\in\mathbb F_p^2,
\]
where every root in $\mathbb F_p$ is counted once. Let $J$ be the
constant matrix of ones on $\mathbb F_p^2$. Additive orthogonality and a one-dimensional
Gauss sum give
\[
 K_\lambda
 =\1_{\{\lambda=0\}}J+
 \frac{G_\eta}{p}\sum_{s\ne0}\eta(-s)
       \chi\left(\frac{\lambda^2}{4s}\right)
       \bigl[\chi(sH(x,y))\bigr]_{x,y}.
\]
Use \eqref{eq:exact-binary-weil-kernel} and put $t=s^{-1}$.
We obtain the identity
\begin{equation}\label{eq:exact-root-arithmetic-transfer}
 K_\lambda-\1_{\{\lambda=0\}}J
 =G_\eta\,\eta(\det R)
   \sum_{t\ne0}\eta(t)\chi\left(\frac{\lambda^2}4t\right)
                  \widetilde\rho(h_t).
\end{equation}
The scalar factor has modulus $\sqrt p$. Therefore, the bound
$2\sqrt p$ from Theorem~\ref{thm:arithmetic-whittaker-curve} gives
$\|K_\lambda\|\le2p$ for $\lambda\ne0$, uniformly in the parameters.

At $\lambda=0$, \eqref{eq:exact-root-arithmetic-transfer} gives the
same bound for $K_0-J$. Indeed, Theorem~\ref{thm:arithmetic-whittaker-curve}
includes $c=0$: the infinity flag has diagonal type $(\eta,1)$,
which excludes residual determinant representations even when the
Whittaker character at $1$ is trivial. We obtain the following corollary.
\begin{corollary}\label{cor:full-root-operator-bound}
For every $p>3$ and every $q\in\mathbb F_p\setminus\{0,1\}$,
including $q=9$, the root matrices satisfy
\[
 \|K_\lambda\|\le2p\quad(\lambda\ne0),\qquad
 \|K_0-J\|\le2p.
\]
\end{corollary}

\subsection{Fourier transform of the rank-one matrix}
\label{sec:rankone-affine-p54}

The degenerate translation orbit leads to the affine matrix
\[
 L_c(H,K)=\sum_{z\ne0}\chi\left(z+
       \frac{c(H^2+HK+K^2-1/3)}z\right),\qquad c\ne0.
\]
We keep the constant $-1/3$ in this definition. The next three
subsections prove $\|L_c\|\le10p$ for every $c\ne0$.

For $B\ne0$, the Sali\'e identity \cite{Salie} is
\[
 \sum_{z\ne0}\eta(z)\chi(Az+B/z)
       =G_\eta\,\eta(B)\sum_{v^2=4AB}\chi(v),
\]
with distinct roots counted once. One direct proof is to take the
unnormalized Fourier transform in $A$. At a nonzero frequency $u$,
the left side transforms to $p\eta(u)\chi(B/u)$. The right side
becomes $G_\eta\,\eta(B)\sum_v\chi(v-uv^2/(4B))$, which has the same
value by completion of the square and $G_\eta^2=\eta(-1)p$.
Both transforms vanish at $u=0$.

Let $\mathcal F_1(K,Y)=p^{-1/2}\chi(KY)$ be the unitary
one-dimensional Fourier matrix. Completing the square in $K$ gives
\begin{align*}
 (L_c\mathcal F_1)(H,Y)
 &=\frac{G_\eta\,\eta(c)}{\sqrt p}\chi(-HY/2)
       \sum_{z\ne0}\eta(z)\chi(A_Yz+B_H/z),\\
 A_Y&=1-Y^2/(4c),\qquad B_H=c(9H^2-4)/12.
\end{align*}
Consequently, whenever $9H^2-4\ne0$,
\[
 (L_c\mathcal F_1)(H,Y)
 =\sqrt p\,\eta\left(-\frac{9H^2-4}{12}\right)\chi(-HY/2)
       \sum_{v^2=(9H^2-4)(4c-Y^2)/12}\chi(v).
\]
Every entry has modulus at most $2\sqrt p$. On the at most two
rows $B_H=0$ the original formula is an elementary Gauss sum or
zero, and the sharper entry bound $\sqrt p$ holds. We use these entry bounds below when adding the omitted rows and
columns.

\subsection{A first-jet Hecke correspondence}
\label{sec:wild-jet-hecke-root}

A first-jet flag is a flag over $\mathbb F_p[X]/(X^2)$.
Fix an odd prime $p$, a nontrivial additive character $\chi$,
$\xi\in\mathbb F_p\setminus\{0,1\}$, and $c\ne0$.
At zero use a first-jet flag, and at $1$ and infinity use ordinary
flags. In degree zero let $E_u=\mathcal O\oplus\mathcal O$, with flags
\[
 (1,uX)\pmod{X^2},\qquad(1,1),\qquad(0,1).
\]
In degree one let $F_U=\mathcal O\oplus\mathcal O(1)$, with flags
\[
 (1,UX)\pmod{X^2},\qquad(1,1),\qquad(1,0),
\]
where the last vector is in the infinity frame. At zero put
\[
 \vartheta(a_0+a_1X+\cdots)=\chi(ca_1/a_0)
\]
on the flag line, and use the trivial character on the quotient
and on the ordinary flags.

\begin{lemma} \label{lem:first-jet-modifications}
For $u,U\notin\{0,1\}$, the elementary modifications
$E_u\hookrightarrow F_U$ at $\xi$ that preserve the specified flags
are, up to a common scalar,
\begin{equation}\label{eq:first-jet-injection}
 G_h(X)=\begin{pmatrix}-h&1\\-hX/z_1&\xi/z_1\end{pmatrix},
 \qquad z_1=\frac{h-\xi}{h-1},
\end{equation}
where $h\in\mathbb F_p\setminus\{0,1,\xi\}$ satisfies
\begin{equation}\label{eq:first-jet-output}
 U h(h-\xi)=(h-1)(h-\xi u).
\end{equation}
Each such $h$ occurs once, and its coefficient on the flag line at
zero is $\chi(-cu/h)$.
Both level objects have automorphism group $\mathbb F_p^\times$,
and distinct parameters give nonisomorphic objects within each family.
These scalars act trivially on the coefficient lines, and the
coordinate sections have the same Petersson mass.
\end{lemma}

\begin{proof}
For $E_u$, the three reduced flag lines are distinct, so their
common stabilizer is the scalar group. For $F_U$, an automorphism is
$\left(\begin{smallmatrix}a&0\\b+dX&e\end{smallmatrix}\right)$;
the infinity, zero, and one flags successively force $d=0$, $b=0$,
and $a=e$. The same calculation applies to an isomorphism between
two objects in either family. Since a scalar does not change the
first-jet parameter, the parameters give distinct isomorphism classes.

The character $\vartheta$ is a character of the compact subgroup
preserving the first-jet flag. It is compatible with a global
central character: with $K=\mathbb F_p(X)$, set
\[
 \omega((a_s)_s)=
 \chi\left(-\sum_s\operatorname{Tr}_{k_s/\mathbb F_p}
      \operatorname{Res}_s\left(\frac cX\frac{da_s}{a_s}\right)\right).
\]
This is trivial on $K^\times$ by the residue theorem, restricts to
$\vartheta^{-1}$ on the scalar units at zero, and is unramified
elsewhere. Thus, the compact type satisfies
$\tau(z)=\omega(z)^{-1}$ on the compact center.
The local type excludes every determinant character: varying the two
diagonal units independently would force the same character to equal
both $\vartheta$ and the trivial character. Constant scalars have
trivial action on the coefficient lines.

The map \eqref{eq:first-jet-injection} acts on the first-jet flag by
\[
 G_h(X)(1,uX)
   =(-h+uX)\left(1,\frac{h-\xi u}{hz_1}X\right)
        \pmod{X^2}.
\]
Thus, its output flag satisfies \eqref{eq:first-jet-output}, and its
coefficient is $\vartheta(-h+uX)=\chi(-cu/h)$.

We verify that the list of modifications is complete. A bundle map
$E_u\to F_U$ has constant first-row entries and affine-linear
second-row entries. Preservation of the infinity flag and of the
reduced zero flag forces the form
\[
 G(X)=\begin{pmatrix}a&b\\dX&e\end{pmatrix}.
\]
The infinity isomorphism gives $b\ne0$, and the zero isomorphism
gives $a,e\ne0$. Divide by the common scalar to put $b=1$ and
write $a=-h$. A length-one cokernel at $\xi$ gives
$-he=d\xi$, while preservation of the flag at $1$ gives
$d+e=1-h$. Solving yields
\[
 d=-\frac{h(h-1)}{h-\xi}=-h/z_1,\qquad
 e=\frac{\xi(h-1)}{h-\xi}=\xi/z_1.
\]
The isomorphism and flag conditions exclude $h=0,1,\xi$.
The first-jet flag condition is \eqref{eq:first-jet-output}.
Conversely, these formulas give a bundle injection with determinant
$(h/z_1)(X-\xi)$, isomorphic at all marked points.
Its kernel at $\xi$ is $\langle(1,h)\rangle$, so distinct $h$ give
distinct elementary modifications. Its image at $\xi$ is
$\langle(1,\xi/z_1)\rangle$, which is also injective in $h$ because
$\xi\ne0,1$. These image lines distinguish the right cosets in
$H_\xi$. A repeated root contributes one coset.

For $u,U\notin\{0,1\}$, the polynomial in
\eqref{eq:first-jet-output} is nonzero at $h=0,1,\xi$. Its projective
root at infinity can occur only when $U=1$. Thus, no modification
is lost on the stated interior indices.
The equal scalar automorphism groups give the coordinate sections
the same Petersson mass, proportional to $(p-1)^{-1}$. The earlier
central-quotient argument applies, so the scalar action is trivial
and both normalized coordinate embeddings are isometric.

For clarity, choose the following local representatives. Write
$S_u(X)=\left(\begin{smallmatrix}1&0\\uX&1\end{smallmatrix}\right)$
and $B=\left(\begin{smallmatrix}1&0\\1&1\end{smallmatrix}\right)$.
Take $g_{E_u,0}=S_u$, $g_{F_U,0}=S_U$, and
$g_{E_u,1}=g_{F_U,1}=B$. At infinity take
$g_{E_u,\infty}=w_0$ and
$g_{F_U,\infty}=\operatorname{diag}(1,X)$; elsewhere take the identity.
For a modification $G_h$, let $a_h$ have component $G_h$ at $\xi$
and the identity elsewhere. Then
\[
 g_{F_U}a_h=G_hg_{E_u}k_h,\qquad k_{h,\xi}=I,\qquad
 k_{h,v}^{-1}=g_{F_U,v}^{-1}G_hg_{E_u,v}\quad(v\ne\xi).
\]
The flag conditions put every $k_{h,v}$ in the specified compact
level. Hence, Lemma~\ref{lem:typed-hecke-compression} assigns to this
coset the coefficient $\tau(k_h)^{-1}$, exactly the forward flag
factor computed above. For a fixed target and $h$, equation
\eqref{eq:first-jet-output} determines $u$ uniquely. Thus the
right-coset multiplicity remains one for the entire coordinate family.
\end{proof}

\subsection{Square parameters in the rank-one matrix}
\label{sec:rankone-square-hecke}

Changing the tame characters in the first-jet model gives the
Sali\'e weight needed for the rank-one matrix at square parameters.

\begin{theorem}\label{thm:rankone-square-hecke}
Let $p>3$, let $c\in\mathbb F_p^{\times2}$, and define
\[
 L_c(H,K)=\sum_{z\ne0}\chi\left(z+
       \frac{c(H^2+HK+K^2-1/3)}z\right).
\]
Then $\|L_c\|\le10p$ in counting measure.
\end{theorem}

\subsubsection*{A weighted first-jet matrix}
Let $\xi\ne0,1$ and $c_0\ne0$. Complete the first-jet correspondence
by the following explicit matrix on all $u,U\in\mathbb F_p$:
\begin{equation}\label{eq:rankone-completed-jet}
 A(u,U)=
 \sum_{\substack{h\in\mathbb F_p\setminus\{0,\xi\}:\\
 U h(h-\xi)=(h-1)(h-\xi u)}}
 \eta\left(\frac h{h-\xi}\right)\chi(-c_0u/h)
       +\1_{U=1}.
\end{equation}
The last term is the projective root $h=\infty$ in the algebraic
completion. The term $h=1$ at $U=0$ is also an algebraic completion;
neither is asserted to be an injection between the same bundle
charts. Only the interior block is identified with a Hecke
compression. We bound the added rows and columns entry by entry.

\begin{lemma}\label{lem:rankone-completed-jet-bound}
The matrix \eqref{eq:rankone-completed-jet} satisfies
$\|A\|\le6\sqrt p$.
\end{lemma}
\begin{proof}
Use the level objects of Lemma~\ref{lem:first-jet-modifications},
with $c=c_0$. Retain the wild character
$\vartheta(a_0+a_1X)=\chi(c_0a_1/a_0)$ on the flagged line at zero.
At $1$ put $\eta$ on the flagged line, and at infinity put $\eta$
on the quotient by the flag. The automorphic central character is
the product of the residue character associated with $-c_0/X$ and
the quadratic tame-symbol character associated with $X-1$ and
$\eta$. It is therefore inverse to the character of the compact
center on the coefficient line.
This type excludes residual determinant characters at zero.
Proposition~\ref{prop:weighted-hecke-bound} therefore bounds the
unnormalized Hecke operator at $\xi$ by $2\sqrt p$.
The two added tame factors have product $\eta(a)^2=1$ on a constant
scalar $a$. Thus, the scalar stabilizers still act trivially, and
the coordinate embeddings from Lemma~\ref{lem:first-jet-modifications}
remain isometric.

For the injection $G_h$ in \eqref{eq:first-jet-injection}, the ordinary
flag line at $1$ has scaling factor $1-h$. With the chosen infinity
representatives, the framed map has reduction
\[
 \left.g_{F_U,\infty}^{-1}G_hg_{E_u,\infty}\right|_{\infty}
 =\begin{pmatrix}1&-h\\0&-h/z_1\end{pmatrix}.
\]
Its map on the quotient line has scaling factor $-h/z_1$, where
$z_1=(h-\xi)/(h-1)$. Their product has quadratic character
\[
 \eta\bigl((1-h)(-h/z_1)\bigr)
       =\eta\left(\frac h{h-\xi}\right).
\]
The first-jet flag at zero gives $\chi(-c_0u/h)$.
Let $A_{\mathrm{int}}$ be the restriction of
\eqref{eq:rankone-completed-jet} to $u,U\notin\{0,1\}$. By
Lemma~\ref{lem:first-jet-modifications} and
Lemma~\ref{lem:typed-hecke-compression},
\begin{equation}\label{eq:rankone-jet-exact-compression}
 i_F^*H_\xi i_E=A_{\mathrm{int}}^{\top}.
\end{equation}
The transpose occurs because $A(u,U)$ has the source parameter as
its row index, whereas the compression has the target as its row
index. The ordinary transpose preserves the norm of a complex
matrix, since $B^{\top}=\overline{B^*}$. Thus,
$\|A_{\mathrm{int}}\|\le2\sqrt p$.

For arbitrary $u,U$ the homogenized quadratic root polynomial is
not identically zero. Indeed, its three coefficients are
$U-1$, $1+\xi u-\xi U$, and $-\xi u$. Their simultaneous vanishing
would imply $U=1,u=0,\xi=1$. It has at most two projective roots.
Thus, every entry of $A$ has modulus at most two, including its
added root at infinity. The two omitted rows and columns have Hilbert--Schmidt norm
at most $4\sqrt p$. Adding the two bounds gives
$\|A\|\le6\sqrt p$.
\end{proof}

\subsubsection*{Fourier transform of the first-jet matrix}
Let $\mathcal F_1(U,y)=p^{-1/2}\chi(Uy)$, and set
\[
 a(u,y)=\frac{(\xi-1)y(1-u)}\xi,
 \qquad b(u,y)=\frac{u(c_0+y)}\xi,
 \qquad C(u,y)=\frac{y+uy(\xi-2)-c_0u}\xi.
\]
We have
\begin{equation}\label{eq:rankone-jet-salie}
 (A\mathcal F_1)(u,y)=
 p^{-1/2}\chi(C(u,y))
       \sum_{s\ne0}\eta(s)\chi\bigl(a(u,y)s+b(u,y)/s\bigr).
\end{equation}
To verify it, sum first over $U$ in
\eqref{eq:rankone-completed-jet} and make the M\"obius substitution
$s=h/(h-\xi)$. The allowed projective $h$ values correspond exactly
to $s\in\mathbb F_p^\times$; in particular $h=\infty$ becomes $s=1$.
Before the substitution the phase is
\[
 y-\frac{u(c_0+y)}h+
       \frac{(\xi-1)y(1-u)}{h-\xi}.
\]
After it, the phase is exactly $C+as+b/s$ and the character is
$\eta(s)$. This calculation is an identity before any estimate is applied.

Writing $G_\eta=\sum_z\chi(z^2)$, the Sali\'e identity gives, when
$b(u,y)\ne0$,
\begin{equation}\label{eq:rankone-jet-root-product}
 (A\mathcal F_1)(u,y)=\frac{G_\eta}{\sqrt p}\eta(b(u,y))\chi(C(u,y))
                \sum_{v^2=4a(u,y)b(u,y)}\chi(v).
\end{equation}
The factor $G_\eta/\sqrt p$ is unimodular.

\begin{proof}[Proof of Theorem~\ref{thm:rankone-square-hecke}]
Choose $k\in\mathbb F_p^\times$ with $k^2=c$. In the preceding
identities put
\[
 \xi=4,\qquad c_0=16k/3,\qquad
 u=\frac12+\frac{3H}4,\qquad y=-\frac{8k}3-\frac{4Y}3.
\]
These are bijective affine changes of the row and column variables.
Direct calculation gives
\[
 4a(u,y)b(u,y)=\frac{(9H^2-4)(4c-Y^2)}{12},
\]
and
\[
 C(u,y)=-HY/2-2kH-2Y/3-2k.
\]
On the other hand, the exact Fourier calculation in
Subsection~\ref{sec:rankone-affine-p54} gives
\[
 (L_c\mathcal F_1)(H,Y)=
 \sqrt p\,\eta\left(-\frac{9H^2-4}{12}\right)\chi(-HY/2)
       \sum_{v^2=(9H^2-4)(4c-Y^2)/12}\chi(v),
\]
provided $9H^2-4\ne0$.
Remove $H=\pm2/3$ and $Y=\pm2k$. On the remaining indices,
$L_c\mathcal F_1$ is exactly $\sqrt p$ times
$A\mathcal F_1$, after the displayed affine coordinate permutations
and row and column factors of modulus one. Indeed, the additional phase
$-2kH-2Y/3-2k$ separates, and the quotient of the quadratic-character
factors is
\[
 \eta\left(\frac{1-u}{3(c_0+y)}\right),
\]
which also separates into row and column factors of modulus one.
The remaining constant $(G_\eta/\sqrt p)^{-1}$ is unimodular and is
included in the constant factor.
Lemma~\ref{lem:rankone-completed-jet-bound} therefore bounds this
interior matrix by $6p$.

The displayed Fourier formula shows that every entry of
$L_c\mathcal F_1$ has modulus at most $2\sqrt p$, including the
exceptional values: when its Sali\'e coefficient $B_H$ vanishes,
the remaining sum is a quadratic Gauss sum or zero.
The two omitted rows and columns have Hilbert--Schmidt norm
at most $4p$.
Since $\mathcal F_1$ is unitary, $\|L_c\|\le10p$ follows.
\end{proof}

\subsection{Nonsquare parameters in the rank-one matrix}
\label{sec:rankone-nonsplit-type}

For nonsquare parameters, we replace the split first-jet character
by a nonsplit semisimple local type. The Hecke point is still
rational. This gives the estimate for all nonzero parameters.

\begin{theorem}\label{thm:rankone-all-parameters}
For $p>3$ and every $c\in\mathbb F_p^\times$, the matrix
\[
 L_c(H,K)=\sum_{z\ne0}\chi\left(z+
       \frac{c(H^2+HK+K^2-1/3)}z\right)
\]
satisfies $\|L_c\|\le10p$. If $c$ is nonsquare, the proof below gives
$\|L_c\|\le(2+\sqrt2)p$.
Consequently, the nonzero rank-one translation blocks have norm
$O(p)$, including the boundary blocks treated in
Subsection~\ref{sec:completed-generic-blocks}.
\end{theorem}

\begin{proof}
The square case is Theorem~\ref{thm:rankone-square-hecke}. We prove
the nonsquare case by a nonsplit local type.

\subsubsection*{The local character and its automorphic Hilbert space}
Let $C\in\mathbb F_p^\times$ be nonsquare and put
\[
 J_C=\begin{pmatrix}0&C\\1&0\end{pmatrix},\qquad
 \mathcal T=\{aI+bJ_C:(a,b)\in\mathbb F_p^2\setminus\{(0,0)\}\}.
\]
This is the nonsplit torus centralizing $J_C$. Every displayed nonzero
matrix is invertible. At the place $X=0$, take the compact subgroup
whose reduction modulo $X$ lies in $\mathcal T$, and work modulo
$X^2$. Every element has a unique form
$B(I+XY)$ modulo $X^2$, with $B\in\mathcal T$ and
$Y\in\operatorname{Mat}_2(\mathbb F_p)$. Define
\begin{equation}\label{eq:rankone-nonsplit-local-character}
 \kappa\bigl(B(I+XY)\bigr)
       =\eta(\det B)\chi(\operatorname{Tr}(J_C Y)).
\end{equation}
This is a character: conjugation by $B$ preserves $J_C$ and hence the
trace pairing. Its restriction to scalar units is trivial, since
$\eta(a^2)=1$ and $\operatorname{Tr}J_C=0$.

At $X=1$ use an ordinary flag with character $\eta$ on its line;
at infinity use an ordinary flag with character $\eta$ on its
quotient. Use maximal compact level elsewhere. Let $\omega$ be the
quadratic tame-symbol character associated with $X-1$. It is
unramified at zero and is its own inverse. Consequently,
$\tau(z)=\omega(z)^{-1}$ on the compact center, so these types
define an automorphic $L^2$ space with the fixed convention above.

No determinant character occurs in this type, irrespective of its
conductor. Indeed, the local element $I+XsE_{12}$ has determinant
exactly one, whereas \eqref{eq:rankone-nonsplit-local-character}
takes the value $\chi(s)$ on it, since
$\operatorname{Tr}(J_C E_{12})=1$. Choosing $s$ with $\chi(s)\ne1$
excludes every representation of the form $\mu\circ\det$.
Proposition~\ref{prop:weighted-hecke-bound} therefore gives
$\|H_\xi\|\le2\sqrt p$ at every rational
$\xi\notin\{0,1\}$.

\subsubsection*{The matrix of the Hecke compression}
In degree zero take $E_H=\mathcal O\oplus\mathcal O$, with the
standard local frame at zero carrying the coefficient
\eqref{eq:rankone-nonsplit-local-character}, and flags
\[
 (1,H)\text{ at }1,\qquad e_2\text{ at infinity}.
\]
In degree one take $F_K=\mathcal O\oplus\mathcal O(1)$, with the
same standard local type at zero, and flags
\[
 (1,K)\text{ at }1,\qquad e_1\text{ at infinity}.
\]
Here $H,K\in\mathbb F_p$, and the target infinity chart divides the
second affine coordinate by $X$.

Both automorphism groups are precisely $\mathbb F_p^\times$.
For $E_H$, an automorphism preserving the local type has reduction
in $\mathcal T$; fixing $e_2$ forces it to be scalar.
For $F_K$, an automorphism has the form
$\left(\begin{smallmatrix}a&0\\b+dX&e\end{smallmatrix}\right)$.
Its infinity flag forces $d=0$, and its constant reduction lies in
$\mathcal T$ only when $b=0$ and $a=e$.
The scalar character is trivial, so the coefficient lines are
well defined on both collections. The same calculations show that
different finite parameters give nonisomorphic objects.
Their full coefficient lines have equal Petersson mass, a common
multiple of $1/(p-1)$ by the central-quotient argument above. Let
$i_E,i_F$ be the corresponding isometric embeddings from counting
measure.

For $t\in\mathbb F_p$ put $D_t=t^2-C\ne0$ and define
\begin{equation}\label{eq:rankone-nonsplit-injection}
 G_t(X)=\begin{pmatrix}
 t&C\\1+D_tX/(C\xi)&t
 \end{pmatrix}.
\end{equation}
It has
\[
 \det G_t(X)=D_t(1-X/\xi),\qquad
 \ker G_t(\xi)=\langle(1,-t/C)\rangle.
\]
Thus, it is an elementary upper modification at $\xi$.
Its value at zero is $tI+J_C\in\mathcal T$, so it identifies the
specified local coefficient fibers. Its infinity flag is $e_1$.
Its flag at $1$ has finite output coordinate
\begin{equation}\label{eq:rankone-nonsplit-output}
 K=\frac{1+D_t/(C\xi)+tH}{t+CH}.
\end{equation}
When $t+CH=0$, the numerator is
$(\xi-1)(1-CH^2)/\xi\ne0$, so the output flag is the omitted
point at infinity rather than a finite $K$.

To see that this list is exhaustive, start with a general bundle map
\[
 G(X)=\begin{pmatrix}a&b\\c+dX&e+fX\end{pmatrix}.
\]
Preservation of the infinity flags forces $f=0$ and $b\ne0$.
Its reduction at zero must belong to $\mathcal T$, so $e=a$ and
$b=Cc$, with $c\ne0$. Dividing by the scalar $c$, which has trivial
total coefficient character, gives
$\left(\begin{smallmatrix}t&C\\1+dX&t\end{smallmatrix}\right)$.
Its determinant has its unique simple zero at $\xi$ precisely when
$d=(t^2-C)/(C\xi)$, giving \eqref{eq:rankone-nonsplit-injection}.
Thus, these formulas account for every elementary modification that
contributes to the stated source and target collections.
The $p$ kernel lines $\langle(1,-t/C)\rangle$ are distinct and
exhaust all kernel lines except $e_2$. The image lines at $\xi$ are
$\langle(C,t)\rangle$, so the corresponding right cosets are also
distinct. The omitted kernel line $e_2$ gives a modification whose
infinity flag lies in the degree-one summand, so it cannot belong to
the target collection, whose infinity flag is transverse to that
summand. For each of the remaining lines,
\eqref{eq:rankone-nonsplit-injection} explicitly realizes the
unique elementary modification. Finite values of its output are
exactly \eqref{eq:rankone-nonsplit-output}; the equal scalar
stabilizers introduce no multiplicity or measure factor.

Take identity representatives at zero,
$g_{E_H,1}=\left(\begin{smallmatrix}1&0\\H&1\end{smallmatrix}\right)$,
and $g_{F_K,1}=\left(\begin{smallmatrix}1&0\\K&1\end{smallmatrix}\right)$.
At infinity take $g_{E_H,\infty}=w_0$ and
$g_{F_K,\infty}=\operatorname{diag}(1,X)$, and use the identity
elsewhere. If $a_t$ is supported at $\xi$ with component $G_t$, then
\[
 g_{F_K}a_t=G_tg_{E_H}k_t,\qquad k_{t,\xi}=I,\qquad
 k_{t,v}^{-1}=g_{F_K,v}^{-1}G_tg_{E_H,v}\quad(v\ne\xi).
\]
These are compact-level elements. For fixed $t$ and a finite output
$K$, the isomorphism $G_t(1)$ determines at most one source flag $H$.
Thus, each right coset contributes once to the full coordinate matrix,
and its coefficient is $\tau(k_t)^{-1}$.

Write $G_t(X)=G_t(0)(I+XY_t)$. Direct multiplication gives
\[
 Y_t=\begin{pmatrix}-1/\xi&0\\t/(C\xi)&0\end{pmatrix},
 \qquad\operatorname{Tr}(J_C Y_t)=t/\xi.
\]
The local coefficient at zero is consequently
$\eta(D_t)\chi(t/\xi)$. The line at $1$ scales by $t+CH$. At
infinity the framed map has reduction
\[
 \left.g_{F_K,\infty}^{-1}G_tg_{E_H,\infty}\right|_{\infty}
 =\begin{pmatrix}C&t\\0&D_t/(C\xi)\end{pmatrix},
\]
so the quotient scales by $D_t/(C\xi)$. Their product is
\[
 \eta(D_t)\eta(t+CH)\eta(D_t/(C\xi))\chi(t/\xi)
 =\eta(C\xi)\eta(t+CH)\chi(t/\xi).
\]
Define
\[
 \mathsf R(H,K)=
 \sum_{t^2+C\xi(H-K)t-C^2\xi HK+C(\xi-1)=0}
      \eta(t+CH)\chi(t/\xi).
\]
Lemma~\ref{lem:typed-hecke-compression} gives the exact identity
\begin{equation}\label{eq:rankone-nonsplit-exact-compression}
 i_F^*H_\xi i_E=\eta(C\xi)\mathsf R^{\top}.
\end{equation}
Here again the displayed matrix uses the source as its row index;
its ordinary transpose has the same norm. Each root is counted once.
Consequently,
\begin{equation}\label{eq:rankone-nonsplit-root-bound}
 \|\mathsf R\|\le2\sqrt p.
\end{equation}

\subsubsection*{Fourier transformation and the affine Kloosterman matrix}
Set $z=t+CH$. Equation \eqref{eq:rankone-nonsplit-output} becomes
\[
 K=\frac z{C\xi}+\frac{(\xi-2)H}{\xi}
                  +\frac{(\xi-1)(1-CH^2)}{\xi z}.
\]
For finite output, $z$ ranges over all of $\mathbb F_p^\times$.
Consequently, with the same unitary Fourier matrix as before,
\[
 (\mathsf R\mathcal F_1)(H,Y)
  =p^{-1/2}\chi\left(\frac{(\xi-2)HY-CH}{\xi}\right)
       \sum_{z\ne0}\eta(z)\chi(az+b/z),
\]
where
\[
 a=\frac{C+Y}{C\xi},\qquad
 b=\frac{Y(\xi-1)(1-CH^2)}{\xi}.
\]
When $Y\ne0$, the exact Sali\'e formula is applicable because
$1-CH^2\ne0$.

Let the Kloosterman parameter $c$ be nonsquare and put
\[
 C=64c/9,\qquad\xi=4,\qquad
 H=-\frac{3Y_0}{16c},\qquad
 Y=-\frac{32c}{9}+\frac{16c}{3}H_0.
\]
Then
\[
 4ab=\frac{(9H_0^2-4)(4c-Y_0^2)}{12},
 \qquad
 \frac{(\xi-2)HY-CH}{\xi}=-H_0Y_0/2+2Y_0/3.
\]
The row and column roles are exchanged by these substitutions.
The resulting root sum is exactly that of $L_c\mathcal F_1$ in
Subsection~\ref{sec:rankone-affine-p54}, divided by $\sqrt p$.
After removing $H_0=\pm2/3$, the ratio of its character factors is
\[
 \eta\left(-\frac{3H_0+2}{4(4c-Y_0^2)}\right),
\]
which separates into row and column factors of modulus one.
The additional phase $2Y_0/3$ depends only on the column, and $G_\eta/\sqrt p$
is unimodular. Since $c$ is nonsquare, $4c-Y_0^2$ never vanishes.
Thus, \eqref{eq:rankone-nonsplit-root-bound} bounds this part of
$L_c\mathcal F_1$ by $2p$, after transposition, coordinate permutations,
and multiplication by unitary diagonal matrices.

On its two omitted rows, the original Sali\'e expression reduces
to a Gauss sum and every entry has modulus at most $\sqrt p$.
Their Hilbert--Schmidt norm is at most $\sqrt2\,p$. This proves
$\|L_c\|\le(2+\sqrt2)p$ for nonsquare $c$.
The square case is Theorem~\ref{thm:rankone-square-hecke}; together
they prove Theorem~\ref{thm:rankone-all-parameters}.
The remaining modes and boundary blocks are treated in
Subsection~\ref{sec:completed-generic-blocks}.
\end{proof}

\subsection{The matrix bound outside the sphere}
\label{sec:completed-generic-blocks}

We now combine the preceding estimates to bound $M_t$ when $Q(t)\ne j$.
We include the boundary blocks and
the centering terms. Norms and convolutions use counting measure.

\begin{theorem}\label{thm:completed-generic-blocks}
Let $p$ be an odd prime, let $Q$ be a split nondegenerate quadratic
form on $\mathbb F_p^4$, and let $S_j=\{Q=j\}$, $j\ne0$.
Then $|S_j|=p^3-p$. Define
\[
 M_t(x,y)=\1_{S_j}(x+y-t)-\frac{|S_j|}{p^4},\qquad x,y\in S_j.
\]
If $Q(t)\ne j$, then
\[
 \|M_t\|\le16p.
\]
For $p>3$ the constant $16$ can be replaced by $11$.
The case $Q(t)=j$ will be treated in
Section~\ref{sec:fourth-moments-sharp-extension}, using the vectors
$a_t$ from Lemma~\ref{lem:product-block-identity}.
\end{theorem}

\begin{proof}
By Lemma~\ref{lem:split-determinant-coordinates}, an invertible linear
change of coordinates identifies $S_j$ with $\mathrm{SL}_2(\mathbb F_p)$
and preserves counting norms and additive convolution. It is therefore
enough to prove the theorem for $Q=\det$ and $j=1$.
Left and right multiplication by determinant-one matrices preserve
$S_j$ and the operator norm. If $\det t=q\ne0$, these actions put
$t=\operatorname{diag}(q,1)$. A nonzero singular $t$ is similarly
equivalent to $E_{12}$; the zero matrix is a separate orbit.

\subsubsection*{Invertible translations: the boundary blocks}
Assume $q\ne0,1$ and let
$S_{\mathrm{bd}}=\{x\in S_j:x_{21}=0\}$ and $S_{\mathrm{big}}=S_j\setminus S_{\mathrm{bd}}$.
An element of $S_{\mathrm{bd}}$ has the form $(a,b;0,a^{-1})$, with $a\ne0$.
An element of $S_{\mathrm{big}}$ is determined by $(A,c,D)$, with $c\ne0$ and
upper-right entry $(AD-1)/c$. Direct determinant expansion says that
the uncentered cross-block $\mathcal B$ has entry one precisely when
\[
 bc=\alpha_a A+\beta_a D+C_a,
 \quad \alpha_a=a^{-1}-1,
 \quad \beta_a=a-q,
 \quad C_a=\alpha_a\beta_a.
\]
The coefficient vector is nonzero. Each row has $p(p-1)$ incidences
and each column has $p-1$. For distinct $a,a'$, the two coefficient
vectors are parallel exactly when
$(a-1)(a'-1)=1-q$; there is at most one distinct partner for each $a$.
Nonparallel rows have $p-1$ common columns. For a parallel pair with
proportionality factor $\kappa$, one has
\[
 C_{a'}-\kappa C_a=\kappa(\kappa-1)C_a\ne0.
\]
Indeed, $\kappa=1$ forces $a'=a$, while $C_a=0$ would give $a=1$
or $a=q$, neither admitting a distinct nonzero partner.
Such rows have $p$ common columns unless $b'=\kappa b$, when they
have none. Rows with the same $a$ and different $b$ have disjoint
supports. Hence,
\[
 \mathcal B\mathcal B^*=(p-1)J+H.
\]
The diagonal $p$-by-$p$ blocks of $H$ are
$p(p-1)I-(p-1)J$. Its only nonzero off-diagonal blocks, which pair
the $a$ labels, have the form $J-pP$ for permutation matrices $P$.
Every such block kills constants and
$\|H\|\le p(p-1)+p=p^2$. It follows that
$\|\mathcal B-p^{-1}J\|\le p$.
Changing the centering to $\frac{|S_j|}{p^4}$ changes only the constant singular
value, to
\[
 (p^{-1}-\frac{|S_j|}{p^4})\sqrt{|S_{\mathrm{bd}}||S_{\mathrm{big}}|}
   =p^{-1/2}(1-p^{-1})<1.
\]
Thus, the centered cross-block has norm at most $p$.

For two points of $S_{\mathrm{bd}}$, incidence is independent of $b,b'$ and is
equivalent to
\[
 (1-a)(a')^2+(a-q)(1-a)a'+a(a-q)=0.
\]
For $a=1$ it has no solution, and otherwise it has at most two
nonzero solutions. The incidence matrix on the $a$ labels is
symmetric and has degree at most two, hence norm at most two.
Tensoring with $J_p$ gives norm at most $2p$. The constant matrix
$\frac{|S_j|}{p^4}J$ on $S_{\mathrm{bd}}$ has norm
$\frac{|S_j|}{p^4}p(p-1)<p$. The centered boundary block
therefore has norm at most $3p$.

\subsubsection*{Fourier multipliers on the big cell}
For $x=(a,b;c,d_x)\in S_{\mathrm{big}}$, use the bijective coordinates
\[
 h=c^{-1},\qquad u=(a+q d_x)/c,\qquad w=(a-q d_x)/c.
\]
They give
\[
 a=\frac{c(u+w)}2,\quad d_x=\frac{c(u-w)}{2q},\quad
 b=\frac{c(u^2-w^2)}{4q}-c^{-1}.
\]
Writing $(u',v,w')$ for the second point, determinant expansion gives
\[
 \det(x+y-t)=1
 \quad\Longleftrightarrow\quad
 (w-w')^2=H((u,h),(u',v)),
\]
where
\[
 H(X,Y)=X^{\top}DX+Y^{\top}DY-2X^{\top}RY,
\quad
 D=\begin{pmatrix}1&0\\0&-4q\end{pmatrix},
\quad
 R=\begin{pmatrix}1&-2q\\-2q&2q(q+1)\end{pmatrix}.
\]
The unitary Fourier transform in $w$ makes the big-cell operator
an orthogonal direct sum of the matrices
\[
 T_\lambda(X,Y)
 =\sum_{z^2=H(X,Y)}\chi(\lambda z)
       -\frac{|S_j|}{p^3}\,\1_{\{\lambda=0\}},\qquad h,v\ne0.
\]
There is no additional power of $p$: these are the unnormalized
Fourier multipliers of convolution in $w$.

Corollary~\ref{cor:full-root-operator-bound} applies to these matrices
for every $q\ne0,1$, including the nonreduced case $q=9$.
It gives $\|K_\lambda\|\le2p$ for $\lambda\ne0$ and
$\|K_0-J\|\le2p$ on all of $\mathbb F_p^2$.
The zero mode has the centering
\[
 T_0=K_0-(1-p^{-2})J=(K_0-J)+p^{-2}J.
\]
Since $J$ is the constant matrix on $p^2$ indices,
$\|p^{-2}J\|=1$, so $\|T_0\|\le2p+1$.
Restriction to $h,v\ne0$ is an orthogonal compression and contracts
the norm. Thus, the whole big-cell block has norm at most $2p+1$.
Combining its two diagonal blocks and the cross-block gives
\[
 \|M_t\|\le\max(3p,2p+1)+p=4p.
\]

\subsubsection*{Nonzero singular translations}
Take $t=E_{12}$ and again separate $S_{\mathrm{bd}}$ and $S_{\mathrm{big}}$ by the lower-left
entry. On $S_{\mathrm{big}}$ write
\[
 x(r,s,h)=
 \begin{pmatrix}r/h&(rs-h^2)/h\\1/h&s/h\end{pmatrix},
 \qquad h\ne0.
\]
The incidence condition is exactly
\[
 (r-r')(s-s')=D_0(h,k),\qquad
 D_0(h,k)=h^2+k^2+hk+h+k.
\]
The double unitary Fourier transform in $r,s$ gives, with no extra
normalization factor,
\[
 T_{\alpha,\beta}(h,k)
 =\sum_{RS=D_0(h,k)}\chi(\alpha R+\beta S)
       -\frac{|S_j|}{p^2}\,\1_{\{\alpha=\beta=0\}}.
\]
If exactly one frequency is zero this is
$p\,\1_{\{D_0=0\}}-1$; if both vanish it is
$p\,\1_{\{D_0=0\}}-1+p^{-1}$.
The matrix $\1_{\{D_0=0\}}$ has at most two ones in each row and
column, so all these modes have norm at most $3p+1$.

If $\alpha\beta=c\ne0$, the multiplier is
\[
 \sum_{z\ne0}\chi\left(z+\frac{cD_0(h,k)}z\right).
\]
For $p>3$, the simultaneous translations $H=h+1/3$, $K=k+1/3$
turn $D_0$ into $H^2+HK+K^2-1/3$.
The affine Kloosterman estimate,
Theorem~\ref{thm:rankone-all-parameters}, bounds the full matrix by
$10p$, for square and nonsquare $c$ alike. The multiplier is
obtained by restricting the rows and columns of this matrix, so it
has the same bound. Consequently, the big-cell block has norm at most
$10p$.

For a boundary row $(a,b;0,a^{-1})$ and big-cell column
$(A,B;C,D)$, incidence is
\[
 (b-1)C=a^{-1}A+aD+1.
\]
Distinct coefficient vectors are parallel only for $a'=-a$, with
factor $-1$; their constant terms then differ by $2\ne0$.
Thus, the preceding row-Gram argument, with $b-1$ in place of $b$,
again gives a centered cross-block of norm at most $p$.
For two boundary points, incidence is
$2+a/a'+a'/a=1$, independent of $b,b'$, with at most two possible
$a'$ for each $a$. Its centered norm is at most $3p$.
Therefore,
\[
 \|M_{E_{12}}\|\le\max(3p,10p)+p=11p.
\]

\subsubsection*{The central translation and the small prime}
For $t=0$, the uncentered operator is convolution on $S_j$ by
the regular semisimple conjugacy class
$\mathcal C=\{g:\operatorname{tr}g=-1\}$.
For $p>3$, let $\epsilon=\eta(-3)$; then
$|\mathcal C|=p(p+\epsilon)$.
Lemma~\ref{lem:sl2-regular-class-bound} gives
$|\mathcal C|\,|\chi_\pi(g)|/\chi_\pi(1)\le2p$ for every nontrivial
irreducible character $\chi_\pi$.
Centering affects only the trivial representation,
whose eigenvalue becomes
\[
 p(p+\epsilon)-\frac{|S_j|}{p^4}(p^3-p)=\epsilon p+2-p^{-2}.
\]
Its modulus is also at most $2p$.
Thus, $\|M_0\|\le2p$.

At the single prime $p=3$, every block has size $|S_j|=24$ and each
entry has modulus at most one. In particular,
$\|M_t\|\le24=8p$ by the row-and-column bound.
This proves the stated bound for every odd prime.
\end{proof}

\section{Fourth moments and the sharp extension estimate}
\label{sec:fourth-moments-sharp-extension}
\label{sec:refined-moment-estimates}

Throughout this section, \(Q\) is split and nondegenerate,
\(S_j=\{Q=j\}\), and \(j\ne0\). Recall that
\(|S_j|=p^3-p\).
All norms on the ambient space and all unmarked norms on \(S_j\) use
counting measure. We write
\[
 Tf(x)=\sum_{\xi\in S_j}f(\xi)\chi(B(x,\xi)),
 \qquad \mathcal F=p^{-2}T,
\]
where functions on \(S_j\) are extended by zero before applying the
unitary ambient Fourier transform \(\mathcal F\).
For the original quadratic form \(Q_4\), the pairing \(B\) is the
usual dot product.

Theorem~\ref{thm:completed-generic-blocks} bounds the matrices for
\(t\notin S_j\). We first prove the two estimates needed when
\(t\in S_j\).

\subsection{A mixed sphere--zero-cone estimate}
\label{subsec:mixed-sphere-zero-cone}

Put \(S_0=\{x:Q(x)=0\}\).

\begin{theorem}\label{thm:mixed-sphere-zero-cone}
For arbitrary complex functions \(f\) supported on \(S_j\) and \(g\)
supported on \(S_0\),
\[
 \|\1_{S_j}(f*g)\|_2^2\le7p\,\|f\|_2^2\|g\|_2^2.
\]
Equivalently, for arbitrary complex \(f,h\) supported on \(S_j\),
\begin{equation}\label{eq:null-correlation-bound}
 \sum_{z\in S_0}
 \left|\sum_y f(y)\overline{h(y+z)}\right|^2
 \le7p\,\|f\|_2^2\|h\|_2^2.
\end{equation}
\end{theorem}

\begin{proof}
For each projective null direction \(v\), choose a nonzero
representative, again denoted by \(v\), and put
\[
 S_v=S_j\cap v^\perp.
\]
The quotient \(v^\perp/\mathbb F_pv\) is a split nondegenerate
binary quadratic space, whose nonzero level has \(p-1\) points.
Thus, \(S_v\) is the disjoint union of \(p-1\) affine lines parallel
to \(v\). For example, in split coordinates \(Q=ad-bc\), the
choice \(v=E_{12}\) gives \(c=0\), \(ad=j\), with \(b\) arbitrary.
Choose an origin \(x_L\) on each of these lines. For
\(\lambda\in\mathbb F_p\), define
\[
 e_{L,\lambda}(x_L+tv)=p^{-1/2}\chi(\lambda t),
 \qquad e_{L,\lambda}=0\quad\hbox{off }L,
\]
and let \(P_{v,\lambda}\) be the orthogonal projection onto the
span of the \(p-1\) vectors \(e_{L,\lambda}\), as \(L\) ranges
over the lines in \(S_v\).
Changing a line origin only multiplies its vector by a unit scalar,
so it does not change the projection.

These projections satisfy
\begin{equation}\label{eq:null-fusion-frame}
 \sum_{v,\lambda}P_{v,\lambda}=(p+1)I_{L^2(S_j)},
 \qquad
 \|P_{v,\lambda}P_{w,\mu}\|_{2\to2}\le p^{-1}
 \quad\hbox{if }(v,\lambda)\ne(w,\mu).
\end{equation}
Indeed, summing over \(\lambda\) gives the coordinate projection
onto \(S_v\). For each \(x\in S_j\), the orthogonal complement
\(x^\perp\) is a nondegenerate ternary quadratic space, whose null
directions form a projective conic with \(p+1\) points. This proves
the first identity.

Distinct frequencies in the same direction are orthogonal. Suppose next
that \(v,w\) are distinct. If \(B(v,w)=0\), then
\(\langle v,w\rangle\) is a maximal totally isotropic plane and
equals its orthogonal complement. Hence, \(S_v\cap S_w\) is empty.
If \(B(v,w)\ne0\), each sphere line parallel to \(v\) meets
\(w^\perp\) at exactly one point. The analogous assertion with
\(v,w\) reversed shows that the intersections of the two line
families form a perfect matching. The overlap matrix
\[
 \bigl(\langle e_{L,\lambda},e_{M,\mu}\rangle\bigr)_{L,M}
\]
is therefore \(p^{-1}\) times a monomial unitary matrix. This proves
the second assertion of \eqref{eq:null-fusion-frame}, with equality
when \(B(v,w)\ne0\).

We next record the resulting fourth-moment bound for the
projections:
\begin{equation}\label{eq:null-projector-fourth}
 \sum_{v,\lambda}\|P_{v,\lambda}f\|_2^4
 \le6\|f\|_2^4.
\end{equation}
Index the projections by \(i\). For any finite collection \(I\), the
block Gram matrix, with diagonal blocks equal to the identity and
off-diagonal blocks of norm at most \(p^{-1}\), gives
\[
 \left\|\sum_{i\in I}P_i\right\|
 \le1+\frac{|I|-1}{p}.
\]
Normalize \(\|f\|_2=1\), and put \(a_i=\|P_if\|_2^2\). Then
\(0\le a_i\le1\) and \(\sum_i a_i=p+1\).
Put \(I_t=\{i:a_i\ge t\}\). Applying the preceding bound to
\(I_t\) gives
\[
 |I_t|\le\frac{1-p^{-1}}{t-p^{-1}}
 \qquad(t>p^{-1}).
\]
In particular, \(|I_{2/p}|\le p-1\), and \(|I_t|\le2/t\) for
\(t\ge2/p\). Splitting at \(2/p\) and integrating the latter bound,
we obtain
\[
 \begin{aligned}
 \sum_i a_i^2
 &\le\frac{2(p+1)}p+
       \frac{4(p-1)}{p^2}
       +\int_{2/p}^1 2t\,\frac2t\,dt\\
 &=6-\frac2p-\frac4{p^2}<6.
 \end{aligned}
\]
Homogeneity proves \eqref{eq:null-projector-fourth}.

To pass to correlations, write
\[
 a_{L,\lambda}=\langle f,e_{L,\lambda}\rangle,\qquad
 b_{L,\lambda}=\langle h,e_{L,\lambda}\rangle.
\]
For one direction, Fourier expansion along its lines gives
\[
 C_v(t):=\sum_{y\in S_v}f(y)\overline{h(y+tv)}
 =\sum_\lambda\chi(-\lambda t)
       \sum_La_{L,\lambda}\overline{b_{L,\lambda}}.
\]
Parseval in \(t\), followed by Cauchy--Schwarz in \(L\), yields
\[
 \sum_t|C_v(t)|^2
 \le p\sum_\lambda
       \|P_{v,\lambda}f\|_2^2\|P_{v,\lambda}h\|_2^2.
\]
For \(t\ne0\), the unrestricted correlation is exactly \(C_v(t)\):
if \(y,y+tv\in S_j\), then \(B(y,v)=0\). Every nonzero vector in
\(S_0\) has exactly one projective direction. Summing over directions
and using \eqref{eq:null-projector-fourth}, we obtain
\[
 \sum_{z\in S_0\setminus\{0\}}
 \left|\sum_y f(y)\overline{h(y+z)}\right|^2
 \le6p\,\|f\|_2^2\|h\|_2^2.
\]
The zero shift contributes at most \(\|f\|_2^2\|h\|_2^2\).
Since \(6p+1\le7p\), this proves
\eqref{eq:null-correlation-bound}.
Finally,
\[
 \langle\1_{S_j}(f*g),h\rangle
 =\sum_{z\in S_0}g(z)\sum_y f(y)\overline{h(y+z)}.
\]
Cauchy--Schwarz in \(z\) and duality in \(h\) prove the convolution
estimate. Conversely, its dual gives the correlation estimate.
\end{proof}

\subsection{Exceptional translations and the fourth moment}
\label{subsec:exceptional-translations-positive-compression}

\begin{lemma}\label{lem:exceptional-regular}
Let \(t\in S_j\), and put
\[
 R_t=\{x\in S_j:Q(x-t)\ne0\}.
\]
For the centered matrix \(M_t\) defined in
\eqref{eq:centered-shell-matrix},
\[
 \|\1_{R_t}M_t\1_{R_t}\|_{2\to2}\le p+2.
\]
\end{lemma}

\begin{proof}
Write \(u=x-t\). By the first identity in
\eqref{eq:tangent-shell-polarity}, each projective direction with
\(Q(u)\ne0\) has at most one nonzero representative satisfying the
sphere condition. Thus, \(x\mapsto[x-t]\) is injective on \(R_t\).
The second identity in \eqref{eq:tangent-shell-polarity} identifies
the restricted incidence matrix with a principal submatrix of
projective polarity on \(\mathbb P^3(\mathbb F_p)\).

Let \(H\) be the full polarity matrix, and let \(J\) be the
all-ones matrix of the same size. Recall that
\(|\mathbb P^3(\mathbb F_p)|=p^3+p^2+p+1\).
Counting the intersections of two projective hyperplanes gives
\[
 H^2=p^2I+(p+1)J,\qquad H\mathbf1=(p^2+p+1)\mathbf1.
\]
It follows that
\[
 \left\|H-\frac{p^2+p+1}{|\mathbb P^3(\mathbb F_p)|}J\right\|=p.
\]
Since
\[
 (p^2+p+1)-\frac{|S_j|}{p^4}|\mathbb P^3(\mathbb F_p)|
 =1+p^{-2}+p^{-3}<2,
\]
we have \(\|H-\frac{|S_j|}{p^4}J\|\le p+2\).
Orthogonal compression to the indicated projective points proves
the lemma.
\end{proof}

\begin{proposition}\label{prop:refined-fourth-moment}
For every complex function \(f\) supported on \(S_j\),
\begin{equation}\label{eq:refined-fourth-moment}
 \|f*f\|_2^2
 \lesssim p\|f\|_2^4+
 p^{-1}\|Tf\|_\infty^2\|f\|_2^2.
\end{equation}
More precisely,
\begin{equation}\label{eq:refined-unitary-fourth}
 \|\mathcal F f\|_4^4
 \le\frac{|S_j|}{p^4}\|\mathcal F f\|_\infty^2\|f\|_2^2
       +Cp^{-3}\|f\|_2^4,
\end{equation}
with an absolute constant \(C\).
\end{proposition}

\begin{proof}
For an arbitrary ambient function \(g\), write
\[
 g=g_0+g_1,\qquad g_0=g\1_{S_0},\qquad g_1=g-g_0.
\]
Apply the product-vector identity
\eqref{eq:product-block-expansion} to \(f,g_1\). When \(t\in S_j\),
its vector
\[
 a_t(x)=f(x)\overline{g_1(x-t)}
\]
is supported on \(R_t\). The bound for \(Q(t)\ne j\) from
Theorem~\ref{thm:completed-generic-blocks} and
Lemma~\ref{lem:exceptional-regular}, together with
\eqref{eq:product-vector-mass}, therefore give
\[
 \|\1_{S_j}(f*g_1)\|_2^2
 \le\frac{|S_j|}{p^4}\|f*g_1\|_2^2
       +C_1p\|f\|_2^2\|g_1\|_2^2.
\]
Convolution by \(f\) has ambient Fourier multiplier \(Tf\), so
\[
 \|\1_{S_j}(f*g_1)\|_2^2
 \le\left(\frac{|S_j|}{p^4}\|Tf\|_\infty^2
                 +C_1p\|f\|_2^2\right)\|g_1\|_2^2.
\]
Theorem~\ref{thm:mixed-sphere-zero-cone} gives
\[
 \|\1_{S_j}(f*g_0)\|_2^2\le7p\|f\|_2^2\|g_0\|_2^2.
\]
The triangle inequality, followed by Cauchy--Schwarz in the two
orthogonal input norms, proves
\begin{equation}\label{eq:positive-convolution-bound}
 \|\1_{S_j}(f*g)\|_2^2
 \le\left(\frac{|S_j|}{p^4}\|Tf\|_\infty^2
              +(C_1+7)p\|f\|_2^2\right)\|g\|_2^2
 \quad\hbox{for every ambient }g.
\end{equation}

Let \(P=\mathcal F\1_{S_j}\mathcal F^*\), the orthogonal projection
onto the Fourier transforms of functions supported on \(S_j\).
We have
\[
 \mathcal F(f*g)=p^2(\mathcal F f)(\mathcal F g).
\]
Dividing \eqref{eq:positive-convolution-bound} by \(p^4\), and
using the unitarity of \(\mathcal F\), gives
\[
 \|PD_{\mathcal F f}\|_{2\to2}^2
 =\|PD_{|\mathcal F f|^2}P\|_{2\to2}
 \le\frac{|S_j|}{p^4}\|\mathcal F f\|_\infty^2
       +Cp^{-3}\|f\|_2^2.
\]
Here \(D_u\) denotes multiplication by \(u\), and the equality
follows by multiplying \(PD_{\mathcal F f}\) by its adjoint.
Since \(P\mathcal F f=\mathcal F f\), taking the inner product
with \(\mathcal F f\) proves
\eqref{eq:refined-unitary-fourth}.
Finally,
\[
 \|f*f\|_2^2=p^4\|\mathcal F f\|_4^4,\qquad
 \|\mathcal F f\|_\infty=p^{-2}\|Tf\|_\infty.
\]
Using \(\frac{|S_j|}{p^4}\le p^{-1}\) proves
\eqref{eq:refined-fourth-moment}.
\end{proof}

\subsection{An orthogonal decomposition}
\label{subsec:coherent-extraction}

The following decomposition will also be used for the localized
estimate.

\begin{lemma}\label{lem:coherent-decomposition}
There is an absolute constant \(A\) with the following property.
Let \(f\) be supported on \(S_j\), with counting norm \(\|f\|_2=1\),
and define
\[
 a_x(\xi)=|S_j|^{-1/2}\chi(-B(x,\xi)),\qquad \xi\in S_j.
\]
There are distinct points \(x_1,\ldots,x_L\) and coefficients \(c_i\)
such that
\[
 f=h+r,\qquad h=\sum_{i=1}^L c_i a_{x_i},\qquad
 L\le\frac{|S_j|}{A^2p^2}\lesssim p,
\]
\[
 \|h\|_2,\|r\|_2\le1,\qquad
 \sum_i|c_i|^2\lesssim1,\qquad
 \sum_i|c_i|^3\lesssim1,\qquad
 \|Tr\|_\infty\le Ap.
\]
The remainder satisfies
\begin{equation}\label{eq:coherent-residual-fourth}
 \|\mathcal F r\|_4^4\lesssim p^{-3}.
\end{equation}
Moreover, \(Th=b+e\), where
\begin{equation}\label{eq:coherent-atomic-main}
 b=\sqrt{|S_j|}\sum_i c_i\1_{\{x_i\}},\qquad
 \|b\|_2^2\lesssim |S_j|,\qquad
 \|b\|_3^3\lesssim |S_j|^{3/2},
\end{equation}
and
\begin{equation}\label{eq:coherent-atomic-tail}
 \|e\|_\infty\lesssim\sqrt p,\qquad
 \|e\|_2\lesssim p^2.
\end{equation}
All constants are uniform over odd primes and nonzero radii.
\end{lemma}

\begin{proof}
Lemmas~\ref{lem:size-kernel} and~\ref{lem:kernel-linfty} give
\[
 \|a_x\|_2=1,\qquad
 |\langle a_x,a_y\rangle|\lesssim p^{-3/2}\quad(x\ne y),
\]
\[
 Ta_x(x)=\sqrt{|S_j|},\qquad |Ta_x(y)|\lesssim1\quad(y\ne x).
\]
These estimates are uniform for all odd primes.

Start with the zero subspace \(W\), and at each stage let
\(r=f-P_Wf\). If \(\|Tr\|_\infty>Ap\), choose a point \(x\) at
which this inequality holds and enlarge \(W\) by \(a_x\).
Every previously selected point satisfies
\[
 Tr(x_i)=\sqrt{|S_j|}\langle r,a_{x_i}\rangle=0,
\]
so the selected points are distinct. Put \(u=a_x-P_Wa_x\).
This vector is nonzero, and the decrease of the squared remainder
norm is
\[
 \frac{|\langle r,u\rangle|^2}{\|u\|_2^2}
 \ge|\langle r,a_x\rangle|^2
 >\frac{A^2p^2}{|S_j|},
\]
because \(\|u\|_2\le1\). Thus, at most
\(|S_j|/(A^2p^2)\) points are selected.

The Gram matrix of the selected functions has diagonal entries one and
off-diagonal row sums bounded by
\[
 CLp^{-3/2}\le\frac{C}{A^2\sqrt p}.
\]
Choose the absolute constant \(A\) so that this is at most \(1/2\)
for every odd prime. The Gram matrix is then uniformly invertible.
At termination, \(h=P_Wf\) and \(r=f-h\), so their norms are at
most one, and the coefficient vector of \(h\) satisfies
\(\sum_i|c_i|^2\lesssim1\).
Consequently, \(\sum_i|c_i|^3\lesssim1\).
The stopping rule gives \(\|Tr\|_\infty\le Ap\).

Proposition~\ref{prop:refined-fourth-moment}, applied to \(r\),
now gives
\[
 \|r*r\|_2^2\lesssim p+p^{-1}(Ap)^2\lesssim p.
\]
The exact convolution normalization implies
\(\|\mathcal F r\|_4^4=p^{-4}\|r*r\|_2^2\), proving
\eqref{eq:coherent-residual-fourth}.

Define \(b\) as in \eqref{eq:coherent-atomic-main}, and put
\(e=Th-b\). Distinctness of the selected points gives the stated
\(L^2\) and \(L^3\) norms of \(b\).
Since \(|Ta_{x_i}(y)|\lesssim1\) for \(y\ne x_i\),
\[
 \|e\|_\infty\lesssim\sum_i|c_i|
 \le\sqrt L\left(\sum_i|c_i|^2\right)^{1/2}
 \lesssim\sqrt p.
\]
Plancherel gives \(\|Th\|_2=p^2\|h\|_2\le p^2\), while
\(\|b\|_2\lesssim p^{3/2}\). Their sum proves the final
\(L^2\) estimate in \eqref{eq:coherent-atomic-tail}.
\end{proof}

\subsection{The strong endpoint and its sharpness}
\label{subsec:strong-endpoint-three}

\begin{proof}[Proof of Theorem~\ref{thm:main-endpoint-three}]
Normalize the counting norm of \(f\) on \(S_j\) to one, and use
Lemma~\ref{lem:coherent-decomposition}.
For \(r\), Plancherel and Cauchy--Schwarz give
\[
 \|\mathcal F r\|_3^3
 \le\|\mathcal F r\|_2\|\mathcal F r\|_4^2
 \lesssim p^{-3/2}.
\]
For \(h\), the same lemma gives
\[
 \|e\|_3^3\le\|e\|_\infty\|e\|_2^2\lesssim p^{9/2},
 \qquad
 \|b\|_3^3\lesssim |S_j|^{3/2}\lesssim p^{9/2}.
\]
Since \(\mathcal F h=p^{-2}(b+e)\), it follows that
\[
 \|\mathcal F h\|_3^3
 \lesssim p^{-6}\bigl(\|b\|_3^3+\|e\|_3^3\bigr)
 \lesssim p^{-3/2}.
\]
The triangle inequality and homogeneity therefore prove
\[
 \|\mathcal F f\|_3\lesssim p^{-1/2}\|f\|_2
\]
for every complex function supported on \(S_j\).

The extension operator normalized by surface probability measure is
\[
 (f\,d\sigma_j)^\vee=|S_j|^{-1}Tf
       =\frac{p^2}{|S_j|}\mathcal F f,
 \qquad
 \|f\|_{L^2(S_j,d\sigma_j)}=|S_j|^{-1/2}\|f\|_2.
\]
Thus,
\[
 \|(f\,d\sigma_j)^\vee\|_3
 \lesssim\frac{p^{3/2}}{\sqrt{|S_j|}}
       \|f\|_{L^2(S_j,d\sigma_j)}
 =(1-p^{-2})^{-1/2}
       \|f\|_{L^2(S_j,d\sigma_j)}.
\]
The factor is uniformly bounded for odd primes. Monotonicity of
counting-measure norms gives the estimate for every \(r\ge3\).
The restriction formulation follows by duality.

For sharpness, choose \(x_0\in S_j\) and a nonzero null vector
\(v\in x_0^\perp\); such a vector exists because this orthogonal
complement is a nondegenerate ternary quadratic space over
\(\mathbb F_p\). Then \(L=x_0+\mathbb F_pv\) is an affine
isotropic line contained in \(S_j\). For \(f=\1_L\), the extension
has modulus \(p/|S_j|\) on the hyperplane \(v^\perp\), of size
\(p^3\), and vanishes elsewhere. Hence,
\[
 \frac{\|(f\,d\sigma_j)^\vee\|_r}
      {\|f\|_{L^2(S_j,d\sigma_j)}}
 =\frac{p^{3/r+1/2}}{\sqrt{|S_j|}}
 \asymp p^{3/r-1}.
\]
This is unbounded as \(p\to\infty\) when \(r<3\), proving that the
range is sharp.
\end{proof}

\section{Localized restriction estimates and the conjecture}
\label{sec:proof-thm-main_conjecture}

The orthogonal decomposition gives a localized estimate stronger than
the one obtained by applying H\"older's inequality to the global
$L^3$ bound.

\begin{proposition}\label{prop:coherent-localized-estimate}
Let $p$ be an odd prime, $j\in\F^\times$, and $A\subseteq\F^4$.
For every complex function $f$ on $S_j$,
\begin{equation}\label{eq:coherent-localized-bound}
 \|(f\,d\sigma_j)^\vee\|_{L^2(A)}^2
 \lesssim
 \left(1+\frac{|A|}{p^2}+\sqrt{\frac{|A|}{p}}\right)
 \|f\|_{L^2(S_j,d\sigma_j)}^2.
\end{equation}
The implied constant is absolute.
\end{proposition}

\begin{proof}
The assertion is immediate when $A$ is empty or $f=0$. Otherwise,
recall that $|S_j|=p^3-p$ and normalize
$\|f\|_{L^2(S_j)}=1$. By Lemma~\ref{lem:coherent-decomposition},
we may write $f=h+r$ and $Th=b+e$, where
\[
 \|\mathcal F r\|_4^4\lesssim p^{-3},\qquad
 \|b\|_2^2\lesssim |S_j|,\qquad
 \|e\|_\infty\lesssim p^{1/2}.
\]
Here $\mathcal F=p^{-2}T$ is unitary. Consequently,
\[
 \|\mathcal F f\|_{L^2(A)}^2
 \lesssim
 p^{-4}\|b\|_2^2+p^{-4}|A|\|e\|_\infty^2
       +|A|^{1/2}\|\mathcal F r\|_4^2
 \lesssim p^{-1}+|A|p^{-3}+|A|^{1/2}p^{-3/2}.
\]
Since $(f\,d\sigma_j)^\vee=(p^2/|S_j|)\mathcal F f$ and
$\|f\|_{L^2(S_j,d\sigma_j)}^2=|S_j|^{-1}$, the ratio of the
squared norms is obtained by multiplying this bound by $p^4/|S_j|$.
This factor is comparable to $p$ and proves
\eqref{eq:coherent-localized-bound}. Homogeneity removes the
normalization.
\end{proof}

\begin{proof}[Proof of Theorem~\ref{thm_main_Lambda}]
Let $A\subseteq\F^4$ with $1\leq |A|\leq p^\vartheta$.
Lemma~\ref{lem:global-auxiliary-bounds} and
$\|g\|_1\leq |A|^{1/2}\|g\|_2$ give
\[
 \|\widehat g\|_{L^2(S_j,d\sigma_j)}^2
 \lesssim\min\{1+|A|p^{-3/2},p\}\|g\|_2^2
 \qquad (\operatorname{supp}g\subseteq A).
\]
Combining this bound, duality, and
Proposition~\ref{prop:coherent-localized-estimate}, we obtain
\begin{equation}\label{eq:localized-four-range-norm}
 \|(f\,d\sigma_j)^\vee\|_{L^2(A)}^2
 \lesssim
 \min\left\{1+|A|p^{-3/2},\,
       1+|A|p^{-2}+\sqrt{|A|/p},\,p\right\}
 \|f\|_{L^2(S_j,d\sigma_j)}^2.
\end{equation}

For $0\leq\vartheta\leq3/2$, the first bound is $O(1)$.
For $3/2\leq\vartheta\leq2$, it is
$O(p^{\vartheta-3/2})$. For $2\leq\vartheta\leq3$, the
second bound is $O(p^{(\vartheta-1)/2})$: in this range both
$1$ and $p^{\vartheta-2}$ are at most
$p^{(\vartheta-1)/2}$. Finally, for $3\leq\vartheta\leq4$
we use Plancherel. Thus, with an absolute constant $C$,
\[
 \|(f\,d\sigma_j)^\vee\|_{L^2(A)}
 \leq C p^{\Lambda_\sharp(\vartheta)}
        \|f\|_{L^2(S_j,d\sigma_j)},
\]
where
\[
 \Lambda_\sharp(\vartheta)=
 \begin{cases}
 0,&0\leq\vartheta\leq3/2,\\[1mm]
 (2\vartheta-3)/4,&3/2\leq\vartheta\leq2,\\[1mm]
 (\vartheta-1)/4,&2\leq\vartheta\leq3,\\[1mm]
 1/2,&3\leq\vartheta\leq4.
 \end{cases}
\]
Given $\varepsilon>0$, choose $p_0(\varepsilon)$ so that
$C\leq p^\varepsilon$ for $p\geq p_0(\varepsilon)$. This proves
the theorem uniformly in $j$. The estimates hold for every complex function $f$.
\end{proof}

%\subsection{The critical support scale} \label{subsec:critical-support-scale}

\begin{proposition}
\label{prop:endpoint-propagation}
Let \(p\) be an odd prime. Suppose that
\[
\sup_{j\in \F^\times}\Lambda_{S_j}(2)\leq\Lambda
\]
for some \(\Lambda\geq0\). Then
\[
\sup_{j\in \F^\times}\Lambda_{S_j}(\vartheta)
\leq
\begin{cases}
\displaystyle
\min\left\{\Lambda,\, \frac{\vartheta}{2},\,
\frac12-\frac12\log_p(1-p^{-2})\right\},
&0\leq\vartheta\leq2,\\[3mm]
\displaystyle
\min\left\{
\Lambda+\frac{\vartheta-2}{2},\,
\frac12-\frac12\log_p(1-p^{-2})
\right\},
&2\leq\vartheta\leq3,\\[3mm]
\displaystyle
\frac12-\frac12\log_p(1-p^{-2}),
&3\leq\vartheta\leq4.
\end{cases}
\]
\end{proposition}
\begin{proof}
Let $E\subseteq\F^4$ with $|E|\leq p^\vartheta$, and let $h:\F^4\to\C$ be supported on $E$.

First, consider the range \(0\leq\vartheta\leq2\). Since
\(|E|\leq p^2\), the definition of $\Lambda_{S_j}(2)$ gives
\begin{equation} \label{eq_case_le2}
\|\widehat{h}\|_{L^2(S_j,d\sigma_j)}
\leq p^{\Lambda_{S_j}(2)}\|h\|_{L^2(\F^4)}\leq p^{\Lambda}\|h\|_{L^2(\F^4)}
\end{equation}
for every $j\in \F^\times$. On the other hand, the trivial restriction
estimate and the Cauchy--Schwarz inequality give
\[
\|\widehat{h}\|_{L^2(S_j,d\sigma_j)}
\leq \|h\|_{L^1(\F^4)}
\leq |E|^{1/2}\|h\|_{L^2(\F^4)}
\leq p^{\vartheta/2}\|h\|_{L^2(\F^4)}.
\]
It follows that $\sup_{j\in \F^\times}\Lambda_{S_j}(\vartheta)
\leq
\min\left\{\Lambda,\,\frac{\vartheta}{2}\right\}$ when $0\leq\vartheta\leq2$.

Second, consider the range \(2\leq\vartheta\leq4\). If \(|E|\leq p^2\), then \eqref{eq_case_le2} still holds. Assume that $|E|>p^2$. Let $B$ be a subset of $E$ chosen uniformly among all subsets of cardinality \(p^2\), and define $h_B:=h\1_B$. Every point of \(E\) belongs to \(B\) with probability \(p^2/|E|\). Therefore,
\[
\mathbb{E} h_B=\frac{p^2}{|E|}h, \qquad \mathbb{E} \widehat{h_B}
=\frac{p^2}{|E|}\widehat{h}.
\]
Since \(\mathbb{E}\) is the finite average over all
\(p^2\)-element subsets \(B\subseteq E\), the triangle inequality and
Cauchy--Schwarz give
\[
\left\|\mathbb{E}_B\widehat{h_B}\right\|_{L^2(S_j,d\sigma_j)}^2
\leq
\mathbb{E}_B
\|\widehat{h_B}\|_{L^2(S_j,d\sigma_j)}^2
\]
for every $j\in \F^\times$. Also, double counting gives $\mathbb{E}_B \|h_B\|_2^2 = \frac{p^2}{|E|}\|h\|_2^2$. Since $\Lambda_{S_j}(2)\leq\Lambda$, we obtain
\[
\left(\frac{p^2}{|E|}\right)^2
\big\|\widehat{h}\big\|_{L^2(S_j,d\sigma_j)}^2
\leq
\mathbb{E}
\|\widehat{h_B}\|_{L^2(S_j,d\sigma_j)}^2 \leq \mathbb{E} \big(p^{2\Lambda}\|h_B\|_{L^2(\F^4)}^2\big)= \frac{p^{2\Lambda+2}}{|E|}\|h\|_{L^2(\F^4)}^2.
\]
Using \(|E|\leq p^\vartheta\), we obtain
\[
\|\widehat{h}\|_{L^2(S_j,d\sigma_j)}
\leq
p^{\Lambda+(\vartheta-2)/2}\|h\|_{L^2(\F^4)}.
\]
Thus, $\sup_{j\in \F^\times}\Lambda_{S_j}(\vartheta)
\leq
\Lambda+\frac{\vartheta-2}{2}$ for $2\leq\vartheta\leq4$.

Finally, the second bound in Lemma~\ref{lem:global-auxiliary-bounds}
gives
\[
 \|\widehat h\|_{L^2(S_j,d\sigma_j)}
 \leq \frac{p^2}{\sqrt{p^3-p}}\|h\|_2
 =p^{\frac12-\frac12\log_p(1-p^{-2})}\|h\|_2.
\]
This estimate holds for every odd prime, uniformly in $j$ and
$\vartheta$. The correction
$-\frac12\log_p(1-p^{-2})$ is $O(p^{-2}/\log p)$.
Combining these bounds completes the proof.
\end{proof}
\begin{remark}
Theorem~\ref{thm_main_Lambda} gives $\Lambda_{S_j}(2)\leq 1/4+\varepsilon$. Using this value in Proposition~\ref{prop:endpoint-propagation} gives weaker bounds than Theorem~\ref{thm_main_Lambda} for some values of $\vartheta$. Proposition~\ref{prop:endpoint-propagation} uses only the bound at $\vartheta=2$. Theorem~\ref{thm_main_Lambda} uses the orthogonal decomposition to estimate each support size directly.
\end{remark}

The following proposition gives the lower bound in Conjecture~\ref{conj_main}. The middle range is a finite-field analogue of the usual Knapp example.

\begin{proposition}
\label{prop:lower-bounds-Lambda}
Let \(p\) be an odd prime. Then $\sup_{j\in \F^\times}\Lambda_{S_j}(\vartheta)
 \geq \Lambda_\flat(\vartheta)$ for all $0\leq \vartheta\leq 4$.
\end{proposition}

\begin{proof}
Take $a,b\in \F$ so that $a^2+b^2=-1$. Define
\[
 u:=(1,a,b,0),\qquad
 w:=(0,-b,a,0),\qquad
 e:=(0,0,0,1).
\]
Then $Q_4(u)=0$, $Q_4(w)=-1$, $Q_4(e)=1$ and $u\cdot w=u\cdot e=w\cdot e=0$. For \(j\in\F_p^\times\), put $s_j:=(1-j)/2$, $t_j:=(1+j)/2$ and $\xi_j:=s_jw+t_je$. These relations give $\xi_j\cdot u=0$, while
\[
 Q_4(\xi_j) =-s_j^2+t_j^2 =\frac{(1+j)^2-(1-j)^2}{4} =j.
\]
Consequently, $Q_4(\xi_j+tu)=Q_4(\xi_j)+2t \xi_j\cdot u+t^2Q_4(u)=j$ for every \(t\in\F\). Thus, $\ell_j:=\xi_j+\F u$ is an isotropic line contained in $S_j$.

Take \(H=u^\perp\), which satisfies $|H|=p^3$ and $H^\perp=\F u$. For any \(A\subseteq H\), define
\[
h_A(x):=\1_A(x)\chi(\xi_j\cdot x).
\]
For every \(t\in\F\),
\[
\widehat{h_A}(\xi_j+tu) =\sum_{x\in A} \chi(\xi_j\cdot x)\chi\big(-(\xi_j +tu)\cdot x\big) =\sum_{x\in A}\chi\big(-t(u\cdot x)\big) = |A|.
\]

Since \(\ell_j\subseteq S_j\), \(|\ell_j|=p\), and \(|S_j|=p^3-p\), it follows that

\begin{equation}
\label{eq:Knapp-ratio-finite-field}
\frac{\|\widehat{h_A}\|_{L^2(S_j,d\sigma_j)}^2}{\|h_A\|_{L^2(\F^4)}^2} \geq \frac{p|A|}{p^3-p}.
\end{equation}

Suppose that \(2\leq\vartheta\leq3\). Choose a set \(A\subseteq H\) with size $\lfloor p^\vartheta\rfloor$. Equation \eqref{eq:Knapp-ratio-finite-field} then gives
\[
\Lambda_{S_j}(\vartheta) \geq \frac{1}{2}\log_p\left( \frac{p\lfloor p^\vartheta\rfloor}{p^3-p}\right)\geq \frac{\vartheta-2}{2}=\Lambda_\flat(\vartheta),
\]
since $\lfloor p^\vartheta\rfloor \geq p^\vartheta-1 \geq p^\vartheta(1-p^{-2})$.

For \(3\leq\vartheta\leq 4\), the set $A=H$ satisfies $|A|=p^3\leq p^\vartheta$. Thus,
\[
\Lambda_{S_j}(\vartheta) \geq \frac{3-2}{2} = \frac{1}{2} =\Lambda_\flat(\vartheta).
\]
Finally, for \(0\leq\vartheta\leq2\), we have $\Lambda_{S_j}(\vartheta)\geq 0= \Lambda_\flat(\vartheta)$.
All the estimates are uniform in \(j\). This completes the proof.
\end{proof}

Propositions~\ref{prop:endpoint-propagation} and \ref{prop:lower-bounds-Lambda} give the following equivalence.

\begin{proposition} \label{prop:conjecture_equiavlence}
Conjecture \ref{conj_main} holds if and only if $\sup_{j\in \F^\times}\Lambda_{S_j}(2)\rightarrow 0$ as $p\rightarrow \infty$.
\end{proposition}

\section{The general restriction/extension transfer}\label{sec:restriction-extension-transfer}

The following general transfer principle follows from a dyadic
decomposition and an $\varepsilon$-removal argument. It gives a strict
range of exponents; the sharp endpoint was proved separately in
Section~\ref{sec:refined-moment-estimates}.

\begin{proposition} \label{thm:restriction}
Suppose that, for every $\varepsilon>0$, there is an integer
$p_1=p_1(\varepsilon)>0$ such that
\[
 \sup_{j\in\F^\times}\Lambda_{S_j}(\vartheta)
 \le\Lambda(\vartheta)+\varepsilon,
 \qquad 0\le\vartheta\le4,
\]
for every prime $p\ge p_1$. Assume also that
$\varphi:=\sup_{0<\vartheta\le4}\Lambda(\vartheta)/\vartheta<1/2$.
Then
\[
 R_{S_j}^*(2\to r)\lesssim_{r,\varphi,p_1}1
 \qquad\text{for every}\qquad r>\frac{2}{1-2\varphi},
\]
uniformly over all odd primes $p$ and all $j\in\F^\times$.
\end{proposition}

\begin{proof}
Fix $\varepsilon>0$ and $p\ge p_1(\varepsilon)$. Let $h$ be supported
on a nonempty set $E\subseteq\F^4$. If $|E|\ge2$, write
$\vartheta=\log_p|E|$. The hypothesis gives
\[
 \Lambda_{S_j}(\vartheta)
 \le\Lambda(\vartheta)+\varepsilon
 \le\varphi\vartheta+\varepsilon.
\]
By the definition of $\Lambda_{S_j}$,
\begin{equation}\label{eq:sparse-restriction-from-Lambda}
 \|\widehat h\|_{L^2(S_j,d\sigma_j)}
 \le p^\varepsilon|E|^\varphi\|h\|_{L^2(\F^4)}.
\end{equation}
For $|E|=1$, the same bound follows from
$|\widehat h(\xi)|=\|h\|_2$ for every $\xi$.

Let $1<t<2/(1+2\varphi)$. We claim that
\begin{equation}\label{eq:q-to-two-restriction}
 \|\widehat g\|_{L^2(S_j,d\sigma_j)}
 \lesssim_{t,\varphi,\varepsilon,p_1}
 p^\varepsilon\|g\|_{L^t(\F^4)}
\end{equation}
for every $g:\F^4\to\mathbb C$, uniformly over all odd primes and
nonzero radii. First suppose that $p\ge p_1(\varepsilon)$.
By homogeneity, we may assume $\|g\|_t=1$, so $|g(x)|\le1$.
Write
\[
 g=\sum_{k\ge0}g_k,\qquad
 g_k=g\1_{A_k},\qquad
 A_k=\{x:2^{-k-1}<|g(x)|\le2^{-k}\}.
\]
Since $|A_k|\le2^{(k+1)t}$, applying
\eqref{eq:sparse-restriction-from-Lambda} to each nonzero $g_k$ gives
\[
 \|\widehat{g_k}\|_{L^2(S_j,d\sigma_j)}
 \le p^\varepsilon2^{-k}|A_k|^{\varphi+1/2}
 \le p^\varepsilon2^{t(\varphi+1/2)}
       2^{-k(1-t(\varphi+1/2))}.
\]
The exponent $1-t(\varphi+1/2)$ is positive. The triangle inequality
and a geometric sum prove \eqref{eq:q-to-two-restriction} for these
primes. For $p<p_1(\varepsilon)$, use the trivial bound
\[
 \|\widehat g\|_{L^2(S_j,d\sigma_j)}
 \le\|g\|_1
 \le p_1(\varepsilon)^{4(1-1/t)}\|g\|_t.
\]
Thus, \eqref{eq:q-to-two-restriction} holds for every odd prime.

Now let $2/(1-2\varphi)<r<\infty$, and choose
$r'<t<2/(1+2\varphi)$, where $r'$ is the conjugate exponent of $r$.
By duality, \eqref{eq:q-to-two-restriction} gives
\[
 R_{S_j}^*(2\to t')
 \lesssim_{t,\varphi,\varepsilon,p_1}p^\varepsilon
\]
for every $\varepsilon>0$.
The sphere kernel formula in Lemma~\ref{lem:size-kernel} and the
decay estimate in Lemma~\ref{lem:kernel-linfty} give
\[
 |(d\sigma_j)^\vee(x)|\lesssim p^{-3/2}\qquad(x\ne0),
\]
uniformly in $j\ne0$. Since $r'<t$ implies $r>t'$, we may apply
Lemma~\ref{lem:epsilon-removal} with increment $r-t'>0$. It follows
that $R_{S_j}^*(2\to r)\lesssim_{r,\varphi,p_1}1$. The case
$r=\infty$ follows directly from Cauchy--Schwarz.
\end{proof}

For $\Lambda=\Lambda_\flat$, a direct calculation gives
$\sup_{0<\vartheta\le4}\Lambda_\flat(\vartheta)/\vartheta=1/6$.
Thus, the proposition recovers the implication from
Conjecture~\ref{conj_main} to the global range $r>3$.

\section{The pinned distance theorem}
\label{subsec:optimized-pinned-distances}

We use the unnormalized Fourier transform and normalized surface
measure. For $t\in\F$, define
\[
 \nu_{E,F}(t):=
 \#\{(x,y)\in E\times F:Q_4(x-y)=t\}.
\]
Then $\sum_{t\in\F}\nu_{E,F}(t)=|E||F|$.

The quadratic Gauss-sum identity in
Lemma~\ref{lem:quadratic-gauss-transform} gives
\begin{equation}\label{eq:zero-sphere-fourier}
 \widehat{\1_{S_0}}(\xi)
 =p^3\1_{\{0\}}(\xi)+p^2\1_{S_0}(\xi)-p.
\end{equation}

\begin{lemma}\label{lem:zero-distance-incidence}
For every $E,F\subseteq\F^4$,
\begin{equation}\label{eq:zero-incidence}
 \nu_{E,F}(0)
 \leq \frac{|E||F|}{p}+p^2|E|^{1/2}|F|^{1/2}.
\end{equation}
\end{lemma}
\begin{proof}
Fourier inversion and \eqref{eq:zero-sphere-fourier} give
\[
 \nu_{E,F}(0)=\frac{|E||F|}{p}
 +p^{-2}\sum_{\xi\in S_0}
       \widehat{\1_E}(\xi)\overline{\widehat{\1_F}(\xi)}
 -p|E\cap F|.
\]
By the Cauchy--Schwarz inequality and Plancherel, the absolute value
of the sum, including its factor $p^{-2}$, is at most
$p^2|E|^{1/2}|F|^{1/2}$. Dropping the last term proves the result.
\end{proof}

An averaged form of the next identity appears in
\cite[Section~5]{chapman}. We separate the term at radius zero and
include the proof with our Fourier normalization.

For fixed $E\subseteq \F^4$ and any $j\in\F$, define
\[
\Phi_j(y)=\sum_{\xi\in S_j}\widehat{\1_E}(\xi)\chi(y\cdot \xi),\quad y\in\F^4.
\]
Also define $\nu_y(j)=|E\cap(y+S_j)|=|\{x\in E\colon Q_4(x-y)=j\}$.

\begin{lemma}
\label{lem:exact-pinned-identity}
For every $y\in\F^4$,
\[
\sum_{j \in \F^\times}\nu_y(j)^2 = \frac{|E|^2}{p} +p^{-4}\sum_{j\in \F^\times}|\Phi_j(y)|^2 +\left(p^{-4}|\Phi_0(y)|^2-\nu_y^2(0)-p^3\1_E(y)\right).
\]
\end{lemma}
\begin{proof}
For $s\in\F$, set
\[
    U_s(y)=\sum_{x\in E}\chi(sQ_4(x-y)).
\]
Orthogonality in the distance variable gives
\[
    \sum_{t\in\F}\nu_y(t)^2
    =\frac{|E|^2}{p}+\frac{1}{p}\sum_{s \in \F^\times}|U_s(y)|^2.
\]
For $s \in \F^\times$, Fourier inversion and the four-dimensional quadratic Gauss-sum
identity give
\[
U_s(y)=p^{-2}\sum_{\xi\in\F^4}\widehat{\1_E}(\xi)\chi(y\cdot \xi) \chi\left(-\frac{Q_4(\xi)}{4s}\right).
\]
Expanding the square and using
\[
\sum_{s \in \F^\times}\chi\left(\frac{Q_4(\xi')-Q_4(\xi)}{4s}\right) =
\begin{cases}p-1,& Q_4(\xi')=Q_4(\xi),\\-1,& Q_4(\xi')\ne Q_4(\xi),\end{cases}
\]
we obtain
\begin{align*}
\frac{1}{p}\sum_{s \in \F^\times}|U_s(y)|^2 =&p^{-5}\sum_{s \in \F^\times}\left|\sum_{\xi\in\F^4}\widehat{\1_E}(\xi)\chi(y\cdot \xi)\chi\left(-\frac{Q_4(\xi)}{4s}\right)\right|^2\\
=&p^{-4}\sum_{j\in\F}|\Phi_j(y)|^2-p^{-5}\left|\sum_{\xi\in\F^4}\widehat{\1_E}(\xi)\chi(y\cdot \xi)\right|^2.
\end{align*}
By Fourier inversion,
\[
\sum_{\xi\in \F^4}\widehat{\1_E}(\xi)\chi(y\cdot \xi)=p^4\1_E(y).
\]
Hence,
\[
\sum_{j\in\F}\nu_y(j)^2 = \frac{|E|^2}{p} +p^{-4}\sum_{j\in\F}|\Phi_j(y)|^2 - p^3\1_E(y).
\]
Subtracting $\nu_y(0)^2$ proves the identity.
\end{proof}

The next lemma bounds the contribution of the zero sphere. We use
\eqref{eq:zero-sphere-fourier} to show the cancellation with
$\nu_y(0)^2$ and the diagonal term.

\begin{lemma}
\label{lem:zero-shell}
For every $E\subseteq \F^4$ and $y\in\F^4$,
\[
p^{-4}|\Phi_0(y)|^2-\nu_y^2(0)-p^3\1_E(y) \le \frac{|E|^2}{p^2}.
\]
\end{lemma}
\begin{proof}
By \eqref{eq:zero-sphere-fourier},
\begin{align*}
\Phi_0(y)
&=\sum_{x\in E}\sum_{\xi\in S_0}\chi((y-x)\cdot \xi)\\
&=p^2\nu_y(0)-p|E|+p^3\1_E(y).
\end{align*}

For $y\notin E$, we have
\[
p^{-4}|\Phi_0(y)|^2-\nu_y^2(0)-p^3\1_E(y) = p^{-2}|E|^2 - 2p^{-1}\nu_y(0)|E|\leq p^{-2}|E|^2.
\]
For $y\in E$, we have
\[
p^{-4}|\Phi_0(y)|^2-\nu_y^2(0)-p^3\1_E(y) =p^{-2}|E|^2+2\nu_y(0)(p-p^{-1}|E|)+p^2-2|E|-p^3.
\]
It remains to show
\begin{equation} \label{eq_lefthandsideneed}
2\nu_y(0)(p-p^{-1}|E|)+p^2-2|E|-p^3\leq 0.
\end{equation}
If $|E|\geq p^2$, then $p-p^{-1}|E|\leq0$. Thus, the left-hand side of
\eqref{eq_lefthandsideneed} is at most $p^2-2|E|-p^3<0$.
If $|E|<p^2$, use $\nu_y(0)\leq|E|$. The same expression is at most
\[
 -2p^3\left(\frac{|E|}{p^2}-\frac12\right)^2
 -\frac{p^3}{2}+p^2-2|E|
 \leq-\frac{p^3}{2}+p^2\leq0,
\]
since $p\geq3$. This completes the proof.
\end{proof}

Define
\[T_y:=\sum_{j\in \F^\times}\nu_y(j)^2, \qquad
\mathcal{I}^\times(E;F) :=\sum_{y\in F}T_y.
\]
Thus, $\mathcal{I}^\times(E;F)$ counts triples $(x,x',y)\in E^2\times F$ with
$Q_4(x-y)=Q_4(x'-y)\ne0$.

The following lemma deduces the pinned distance conclusion from two
mean estimates. Similar arguments appear in
\cite[proof of Theorem~2.3]{chapman} and \cite[Lemma~2.2]{KohPinnedNote}.
We include the proof to account for the zero distance.

\begin{lemma}
\label{lem:almost-every-pin-energy-criterion}
Let \(E,F\subseteq\F^4\) be non-empty. Suppose that there are quantities \(\varepsilon_p,\rho_p\ge0\), with $\varepsilon_p\rightarrow 0$ and $\rho_p\longrightarrow 0$ as $p\rightarrow \infty$, such that
\begin{equation}
\label{eq:almost-every-pin-IEF-hypothesis}
\mathcal{I}^\times(E;F) \le (1+\varepsilon_p)\frac{|E|^2|F|}{p}
\end{equation}
and
\begin{equation}
\label{eq:almost-every-pin-zero-hypothesis}
\frac{\nu_{E,F}(0)}{|E||F|} \le \rho_p.
\end{equation}
Then there exists \(F'\subseteq F\) such that $|F'|=(1+o(1))|F|$
and
\[
\#\big(\{Q_4(x-y):x\in E\}\setminus\{0\}\big) =(1+o(1))p
\]
uniformly for every \(y\in F'\).
\end{lemma}
\begin{proof}
By \eqref{eq:almost-every-pin-IEF-hypothesis}, we have
\[
\frac{p}{|E|^2}\sum_{y\in F}T_y=\frac{p}{|E|^2}\mathcal{I}^\times(E;F) \le (1+\varepsilon_p)|F|.
\]
Write $\tau_p:=(\rho_p+p^{-1})^{1/2}$. Then \(\tau_p=o(1)\). Define $F_0:=\{y\in F:\nu_y(0) \le \tau_p|E|\}$. Also note that $\nu_{E,F}(0)=\sum_{y\in F}\nu_y(0)$. By Markov's inequality and \eqref{eq:almost-every-pin-zero-hypothesis},
\[
|F\setminus F_0| \leq \frac{1}{\tau_p |E|}\sum\limits_{y\in F}\nu_y(0) \le \frac{\rho_p}{\tau_p}|F| \leq \tau_p |F|.
\]
For \(y\in F_0\), the Cauchy--Schwarz inequality gives
\[
T_y \ge \frac{(|E|-\nu_y(0) )^2}{p-1} \ge (1-2\tau_p)\frac{|E|^2}{p}.
\]
Combining these bounds gives
\[
\begin{aligned}
&\sum_{y\in F_0}\max\Big\{\frac{p T_y}{|E|^2}-1,\, 0\Big\} \le \sum_{y\in F_0} \Big(\frac{p T_y}{|E|^2}-1\Big) + 2\tau_p|F_0|\\
&\qquad \leq \sum\limits_{y\in F}\frac{p T_y}{|E|^2} -|F_0| +2\tau_p |F| \leq (\varepsilon_p+2\tau_p)|F|+|F\setminus F_0|\le
(\varepsilon_p+3\tau_p)|F|.
\end{aligned}
\]

Set $\delta_p:=(\varepsilon_p+3\tau_p)^{1/2}=o(1)$. By Markov's inequality, the set
\[
H:=\{y\in F_0:\, T_y>(1+\delta_p)p^{-1}|E|^2\} = \Bigg\{y\in F_0:\, \max\Big\{\frac{p T_y}{|E|^2}-1,\, 0\Big\} >\delta_p \Bigg\}
\]
satisfies
\[
|H|\le \frac{(\varepsilon_p+3\tau_p)|F|}{\delta_p} =\delta_p |F|=o(|F|).
\]

Now let \(F':=F_0\setminus H\). Then $|F'|=(1+o(1))|F|$. For every \(y\in F'\), another application of Cauchy--Schwarz gives
\[
\begin{aligned}
\#\big(\{Q_4(x-y):x\in E\}\setminus\{0\}\big) \ge
\frac{(|E|-\nu_y(0) )^2}{T_y} \ge
\frac{(1-\tau_p)^2|E|^2}{(1+\delta_p)p^{-1}|E|^2} =(1+o(1))p,
\end{aligned}
\]
by the definitions of $F_0$ and $H$. Since the left-hand side is at most \(p-1\), the conclusion follows.
\end{proof}
\begin{proof}[Proof of Theorem~\ref{thm:almost-every-asymmetric-pinned}]
Lemmas~\ref{lem:exact-pinned-identity} and \ref{lem:zero-shell} imply
\begin{equation}
\label{eq:optimized-pinned-energy-start}
\mathcal{I}^\times(E;F)
\leq
\frac{|E|^2|F|}{p}
+p^{-4}\sum_{j\in\F^\times}\sum_{y\in F}|\Phi_j(y)|^2
+\frac{|E|^2|F|}{p^2}.
\end{equation}
For \(j\neq0\), we have $\Phi_j(y)
=|S_j|(\widehat{\1_E}\,d\sigma_j)^\vee(y)$ and $|S_j|=p^3-p$. By the definition of $\Lambda_{S_j}(\vartheta)$,
\begin{align*}
p^{-4}\sum_{j\in\F^\times}\sum_{y\in F}|\Phi_j(y)|^2 =
p^{-4}\sum_{j\in\F^\times}|S_j|^2
\|(\widehat{\1_E}\,d\sigma_j)^\vee\|_{L^2(F)}^2 \lesssim
\sum_{j\in\F^\times} p^{2+2\Lambda_{S_j}(\vartheta)}
\|\widehat{\1_E}\|_{L^2(S_j,d\sigma_j)}^2.
\end{align*}
Note that $\sup_{j\in\F^\times}\Lambda_{S_j}(\vartheta)\leq \Lambda(\vartheta)$ $(0\leq \vartheta\leq 4)$ for $p\geq p_3$. Moreover, Plancherel's theorem gives
\begin{align*}
\sum_{j\in\F^\times}
\|\widehat{\1_E}\|_{L^2(S_j,d\sigma_j)}^2
=\frac{1}{p^3-p}
\sum_{\substack{\xi\in\F^4\\Q_4(\xi)\neq0}}
|\widehat{\1_E}(\xi)|^2
\lesssim p|E|.
\end{align*}
Therefore,
\begin{equation}\label{distance_use_later}
p^{-4}\sum_{j\in\F^\times}\sum_{y\in F}|\Phi_j(y)|^2
\lesssim
p^{3+2\Lambda(\vartheta)}|E|.
\end{equation}
Since $|E||F|\geq p^{4+2\Lambda(\vartheta)+\varepsilon}$, we have
\[
\frac{p^{3+2\Lambda(\vartheta)}|E|}{|E|^2|F|/p} = \frac{p^{4+2\Lambda(\vartheta)}}{|E||F|} =o_\varepsilon(1).
\]
Hence, \eqref{eq:optimized-pinned-energy-start} gives
\[
\mathcal{I}^\times(E;F)\leq(1+o_\varepsilon(1))\frac{|E|^2|F|}{p}.
\]

By Lemma~\ref{lem:zero-distance-incidence}, the condition $|E||F|\geq p^{4+2\Lambda(\vartheta)+\varepsilon}$ also ensures that
\[
\begin{aligned}
\frac{\nu_{E,F}(0)}{|E||F|} &\lesssim
\frac{1}{p}+\frac{p^2}{|E|^{1/2}|F|^{1/2}} =o_\varepsilon(1).
\end{aligned}
\]
Lemma~\ref{lem:almost-every-pin-energy-criterion} gives a subset
\(F'\subseteq F\), with
\(|F'|=(1+o_\varepsilon(1))|F|\), such that
\[
|\Delta_y(E)\setminus\{0\}| = (1+o_\varepsilon(1))p
\]
uniformly for every \(y\in F'\). This completes the proof.
\end{proof}

\begin{proof} [Proof of Corollary \ref{cor_distance_1}]
Given $\varepsilon>0$, suppose that $|F|\leq |E|$ and $|E||F| \geq p^{4+\varepsilon}$. Choose $\Lambda(\vartheta)=\Lambda_\flat(\vartheta)+\varepsilon/8$ $(0\leq \vartheta\leq 4)$.
By Conjecture~\ref{conj_main}, we have $\sup_{j\in \F^\times}\Lambda_{S_j}(\vartheta)\leq \Lambda(\vartheta)$ for all $0\leq \vartheta\leq 4$, when \(p\) is sufficiently large depending on $\varepsilon$.

Take $\vartheta$ such that $|F|=p^\vartheta$, and write $|E|=p^{\vartheta'}$. Then
\(\vartheta\leq \vartheta'\) and \(\vartheta+\vartheta'\geq 4+\varepsilon\), which implies \(\vartheta'\geq 2+\varepsilon/2\). We claim that
\begin{equation} \label{eq_coro_verify}
|E||F|= p^{\vartheta+\vartheta'} \geq p^{4+2\Lambda(\vartheta)+\varepsilon/4},\qquad (0\leq \vartheta\leq 4).
\end{equation}

Indeed, if $0\leq \vartheta\leq 2$, then $4+2\Lambda(\vartheta)+\varepsilon/4=4+\varepsilon/2\leq \vartheta+\vartheta'$.

If \(2\leq\vartheta\leq3\), then $4+2\Lambda(\vartheta)+\varepsilon/4
=\vartheta+2+\varepsilon/2 \leq \vartheta+\vartheta'$.

If \(3\leq\vartheta\leq4\), then $4+2\Lambda(\vartheta)+\varepsilon/4=3+(2+ \varepsilon/2) \leq \vartheta+\vartheta'$.

This proves \eqref{eq_coro_verify}. The conclusion follows from Theorem~\ref{thm:almost-every-asymmetric-pinned} with $\varepsilon/4$.
\end{proof}

The four branches of Theorem~\ref{thm_main_Lambda} give the following
consequence.

\begin{corollary}\label{cor_distance_concrete}
Let $\varepsilon>0$, and let $E,F\subseteq\F^4$ with
$F\ne\varnothing$. Suppose that one of the following conditions holds:

(\romannumeral1) $1\leq|F|\leq p^{3/2}$ and
$|E||F|\geq p^{4+\varepsilon}$;

(\romannumeral2) $p^{3/2}\leq|F|\leq p^2$ and
$|E|\geq p^{5/2+\varepsilon}$;

(\romannumeral3) $p^2\leq|F|\leq p^3$ and
$|E||F|^{1/2}\geq p^{7/2+\varepsilon}$;

(\romannumeral4) $p^3\leq|F|\leq p^4$ and
$|E||F|\geq p^{5+\varepsilon}$.

Then there exists $F'\subseteq F$ such that
$|F'|=(1+o_\varepsilon(1))|F|$ and
$|\Delta_y(E)\setminus\{0\}|=(1+o_\varepsilon(1))p$ uniformly
for every $y\in F'$.
\end{corollary}

\begin{proof}
For nonempty $E,F$, put $\vartheta=\log_p|F|$. By
\eqref{eq:localized-four-range-norm},
\[
 \|(f\,d\sigma_j)^\vee\|_{L^2(F)}^2
 \lesssim p^{2\Lambda_\sharp(\vartheta)}
             \|f\|_{L^2(S_j,d\sigma_j)}^2,
\]
with an absolute constant. Substituting this in
\eqref{eq:optimized-pinned-energy-start} and using Plancherel as in
\eqref{distance_use_later}, we obtain the quantitative bound
\begin{equation}\label{eq:quantitative-pinned-energy}
 \frac{\mathcal I^\times(E;F)}{|E|^2|F|/p}
 \leq 1+C\frac{p^{4+2\Lambda_\sharp(\vartheta)}}{|E||F|}
          +\frac1p.
\end{equation}
The four branches give
\[
 4+2\Lambda_\sharp(\vartheta)=
 \begin{cases}
 4,&0\leq\vartheta\leq3/2,\\[1mm]
 5/2+\vartheta,&3/2\leq\vartheta\leq2,\\[1mm]
 7/2+\vartheta/2,&2\leq\vartheta\leq3,\\[1mm]
 5,&3\leq\vartheta\leq4.
 \end{cases}
\]
Thus, each stated hypothesis is equivalent to
\begin{equation}\label{eq:four-range-common-condition}
 |E||F|\geq p^{4+2\Lambda_\sharp(\vartheta)+\varepsilon}.
\end{equation}
The error in \eqref{eq:quantitative-pinned-energy} is at most
$Cp^{-\varepsilon}+p^{-1}$. Since $\Lambda_\sharp\geq0$,
\eqref{eq:zero-incidence} also gives
\[
 \frac{\nu_{E,F}(0)}{|E||F|}
 \lesssim p^{-1}+p^{-\varepsilon/2}=o_\varepsilon(1).
\]
Lemma~\ref{lem:almost-every-pin-energy-criterion} proves the conclusion,
with the zero distance excluded.
\end{proof}

\begin{proof}[Proof of Corollary~\ref{cor_distance_2}]
The four branches satisfy
$2\Lambda_\sharp(\vartheta)\leq\vartheta/3$ for
$0\leq\vartheta\leq4$. Hence, with $\vartheta=\log_p|F|$,
\[
 |E||F|
 =|E||F|^{2/3}p^{\vartheta/3}
 \geq p^{4+\vartheta/3+\varepsilon}
 \geq p^{4+2\Lambda_\sharp(\vartheta)+\varepsilon}.
\]
This is \eqref{eq:four-range-common-condition}, so
Corollary~\ref{cor_distance_concrete} applies with the same $\varepsilon$.
\end{proof}

\begin{remark}
The conclusion of Corollary~\ref{cor_distance_2} also holds whenever
$|E||F|^{2/3}/p^4\to\infty$. Indeed,
$2\Lambda_\sharp(\vartheta)\leq\vartheta/3$ and
\eqref{eq:quantitative-pinned-energy} bound the relative
second-moment error by
\[
 C\frac{p^4}{|E||F|^{2/3}}+\frac1p=o(1).
\]
Also, $|E||F|/p^4\to\infty$, so \eqref{eq:zero-incidence} gives
$\nu_{E,F}(0)/(|E||F|)=o(1)$.
Lemma~\ref{lem:almost-every-pin-energy-criterion} applies.
\end{remark}

\begin{proof}[Proof of Corollary~\ref{cor:improved-middle-pins}]
This is case \textup{(iii)} of Corollary~\ref{cor_distance_concrete},
with the same $\varepsilon$.
\end{proof}

\begin{proof}[Proof of Corollary~\ref{cor:same-set-seven-thirds}]
We prove the stronger conclusion stated after the corollary.
Assume that $|E|/p^{7/3}\to\infty$ and put $F=E$.
For sufficiently large $p$, we have $|E|\geq p^2$.
Since $2\Lambda_\sharp(\vartheta)\leq(\vartheta-1)/2$ for
$2\leq\vartheta\leq4$, \eqref{eq:quantitative-pinned-energy} gives
\[
 \frac{\mathcal I^\times(E;E)}{|E|^3/p}
 \leq1+C\frac{p^{7/2}}{|E|^{3/2}}+\frac1p=1+o(1).
\]
By \eqref{eq:zero-incidence},
\[
 \frac{\nu_{E,E}(0)}{|E|^2}
 \lesssim p^{-1}+\frac{p^2}{|E|}=o(1).
\]
Lemma~\ref{lem:almost-every-pin-energy-criterion} proves the assertion.
The bounds are uniform for $|E|/p^{7/3}\geq\omega(p)$ whenever
$\omega(p)\to\infty$. In particular, they are $o_\varepsilon(1)$
under the stated hypothesis $|E|\geq p^{7/3+\varepsilon}$.
\end{proof}

\section{Further remarks}
\label{section_app}

We have proved the global spherical restriction conjecture in dimension four, including its strong endpoint. The localized conjecture remains open. Theorem~\ref{thm_main_Lambda} gives $\Lambda_{S_j}(2)\leq1/4+\varepsilon$, whereas Conjecture~\ref{conj_main} predicts convergence to zero. The conjectural distance bound for $|E|\geq p^{2+\varepsilon}$ is therefore stronger than Corollary~\ref{cor:same-set-seven-thirds}.
The examples in Proposition~\ref{prop:lower-bounds-Lambda} give necessary
lower bounds. They do not establish the conjectured upper bounds.

We conclude with the localized restriction conjecture in even dimensions over prime fields. Let \(d=2m\geq 2\) be fixed. As before, \(p\) is an odd prime, \(\F=\F_p\), and
\[
 S_j^{(d)}:=\{\xi\in\F^d:Q_d(\xi)=j\},
 \qquad j\in\F^\times.
\]
We write \(d\sigma_j^{(d)}\) for the normalized surface measure on
\(S_j^{(d)}\). For \(0\leq\vartheta\leq d\), define
\(\Lambda_{S_j^{(d)}}(\vartheta)\) to be the infimum of all
\(\lambda\geq0\) such that, for every \(E\subseteq\F^d\) with
\(|E|\leq p^\vartheta\) and every function \(g:\F^d\to\C\) supported
on \(E\),
\[
 \|\widehat{g}\|_{L^2(S_j^{(d)},d\sigma_j^{(d)})}
 \leq p^\lambda\|g\|_{L^2(\F^d)}.
\]

Define
\[
\Lambda_\flat^{(d)}(\vartheta):=
\begin{cases}
0, &0\leq\vartheta\leq m,\\[1mm]
\dfrac{\vartheta-m}{2}, &m\leq\vartheta\leq m+1,\\[2mm]
\dfrac{1}{2}, &m+1\leq\vartheta\leq2m.
\end{cases}
\]

The higher-dimensional version of Conjecture~\ref{conj_main} is the following.

\begin{conjecture}
\label{conj:higher-even-dimensional-localized}
For every fixed even integer \(d=2m\geq 2\), we have
\[
\sup_{j\in \F^\times}\Lambda_{S_j^{(d)}}(\vartheta)\rightarrow \Lambda_\flat^{(d)}(\vartheta),\qquad p\rightarrow \infty,
\]
uniformly for all $0\leq \vartheta\leq 2m$.
\end{conjecture}

The localized conjecture would imply $R_{S_j^{(d)}}^*(2\rightarrow r)\lesssim_r1$ for every $r>\frac{2(d+2)}{d}$. The transfer argument gives only this strict range. The localized conjecture remains open in dimension two, dimension four, and higher even dimensions.

\begin{remark}\label{rem:localized-conjecture-distance-implications}
Conjecture~\ref{conj_main} implies the Erd\H{o}s--Falconer distance
conclusion in dimension four. Indeed, apply
Corollary~\ref{cor_distance_1} with $F=E$. For each fixed
$\varepsilon>0$, the hypothesis $|E|\geq p^{2+\varepsilon}$ then
gives $(1-o_\varepsilon(1))p$ distances from all but
$o_\varepsilon(|E|)$ pins in $E$.

In dimension two, the critical case $\vartheta=1$ of
Conjecture~\ref{conj:higher-even-dimensional-localized} already gives
the corresponding conclusion for $|E|\geq p^{1+\varepsilon}$.
Fix $\varepsilon>0$ and take $p$ sufficiently large that
$\sup_{j\ne0}\Lambda_{S_j^{(2)}}(1)\leq\varepsilon/4$.
For any $A\subseteq\F^2$ with $|A|\geq p$, duality and averaging over
all $p$-point subsets of $A$ give
\[
 \|(f\,d\sigma_j^{(2)})^\vee\|_{L^2(A)}^2
 \leq p^{\varepsilon/2}\frac{|A|}{p}
       \|f\|_{L^2(S_j^{(2)},d\sigma_j^{(2)})}^2.
\]
Define $\nu_y(j)$ and $\Phi_j(y)$ as in Section~\ref{subsec:optimized-pinned-distances},
with $Q_2$ and $S_j^{(2)}$ in place of $Q_4$ and $S_j$.
The planar Gauss-sum formula gives
\[
 \Phi_0(y)=p\1_E(y)+\eta(-1)\big(p\nu_y(0)-|E|\big),\qquad
 p^{-2}|\Phi_0(y)|^2-\nu_y(0)^2-p\1_E(y)
 \leq\frac{|E|^2}{p^2}
\]
when $|E|\geq p$. For $\eta(-1)=-1$, one has
$\nu_y(0)=\1_E(y)$ and $\Phi_0(y)=|E|$.
For $\eta(-1)=1$, expansion of the square proves the inequality since
$\1_E(y)-|E|/p\leq0$.
Thus, the calculation in Lemma~\ref{lem:exact-pinned-identity},
the preceding localized estimate, $|S_j^{(2)}|=p-\eta(-1)$ for $j\ne0$,
and Plancherel give
\[
 \frac{\displaystyle\sum_{y\in E}\sum_{j\ne0}\nu_y(j)^2}
 {|E|^3/p}
 \leq1+C\frac{p^{1+\varepsilon/2}}{|E|}+\frac1p.
\]
Also, $|S_0^{(2)}|\leq2p-1$, so
$\sum_{y\in E}\nu_y(0)\leq(2p-1)|E|$.
If $|E|\geq p^{1+\varepsilon}$, the relative energy error is
$O(p^{-\varepsilon/2}+p^{-1})$, and the proportion of zero-distance
pairs is $O(p^{-\varepsilon})$. The argument of
Lemma~\ref{lem:almost-every-pin-energy-criterion} therefore gives
$(1-o_\varepsilon(1))p$ nonzero distances from all but
$o_\varepsilon(|E|)$ pins.
\end{remark}

\begin{remark}\label{rem:ford-constant-factor-obstruction}
The hypothesis $|E|\geq p^{2+\varepsilon}$ in the conjectured distance
conclusion cannot be replaced by $|E|\geq Cp^2$ for a fixed constant $C$. Indeed, Lemma~\ref{lem:split-determinant-coordinates} allows us to use
coordinates in which $Q_4(u,v,z,w)=uv+zw$. For a fixed $C>0$, put
$m=\lceil\sqrt{Cp}\rceil$, $I_m=\{1,\ldots,m\}\subseteq\F_p$,
and
\[
 E=\F_p\times\{0\}\times I_m\times I_m.
\]
For sufficiently large $p$, one has $m<p$ and $|E|=pm^2\geq Cp^2$.
Every distance is zero or is congruent modulo $p$ to $\pm ab$ for some
$1\leq a,b\leq m-1$. Ford's multiplication-table estimate
\cite[Corollary~3]{Ford2008} gives
\[
 \big|\{ab:a,b\in\mathbb Z,\ 1\leq a,b\leq k\}\big|\asymp
 \frac{k^2}{(\log k)^\delta(\log\log k)^{3/2}},
\]
where $\delta=1-(1+\log\log2)/\log2>0$. Therefore,
\[
 |\Delta(E)|
 \leq2\big|\{ab:a,b\in\mathbb Z,\ 1\leq a,b\leq m-1\}\big|+1
 =o(p).
\]
Thus, $|E|\geq Cp^2$ does not ensure a positive proportion of all
distances, for any fixed $C$. This example does not contradict the
conjectured conclusion under $|E|\geq p^{2+\varepsilon}$.
\end{remark}

\section*{Acknowledgements}
We thank Doowon Koh for helpful discussions and for sharing his ideas.
The second author also thanks Danqing He and Shanlin Huang for discussions
on the Mizohata--Takeuchi conjecture.

\end{document}